\documentclass[11pt]{amsart}
\usepackage{amsmath,amsfonts,amssymb,amsthm,amscd}
\usepackage{amsfonts,amsmath,amssymb}
\usepackage{enumerate}
 \usepackage{xcolor}
\usepackage{bm}
\usepackage{tikz}
\usepackage{hyperref}
\usepackage{stmaryrd}

\newcommand{\R}{\mathbb{R}}

\newcommand{\eps}{\epsilon}

\newcommand{\Ric}{\mathrm{Ric}}
\renewcommand{\CD}{\mathrm{CD}}
\newcommand{\LSI}{\mathrm{LSI}}
\newcommand{\var}{\mathrm{Var}}

\newcommand{\abs}[1]{\left | #1\right |}
\newcommand{\norm}[1]{\left\Vert #1 \right\Vert}
\newcommand{\Hess}{\mathrm{Hess}}
\newcommand{\mani}{\mathcal{M}}
\newcommand{\rr}{\mathbb{R}}
\newcommand{\iprod}[1]{\left\langle #1 \right\rangle}
\newcommand{\wass}{\mathbb{W}}
\newcommand{\KL}{\mathrm{KL}}
\newcommand{\vol}{\mathrm{vol}}

\newcommand{\E}{\mathrm{E}}
\newcommand{\Cov}{\mathrm{Cov}}
\newcommand{\Var}{\mathrm{Var}}
\newcommand{\Tr}{\mathrm{Tr}}
\renewcommand{\div}{\mathrm{div}}
\newcommand{\genmld}{\mathcal{L}}
\newcommand{\cdc}{\mathcal{E}}

\usepackage{amsmath}

\DeclareMathOperator*{\argmin}{arg\,min}

\newtheorem{lemma}{Lemma}

\newtheorem{theorem}{Theorem}

\newtheorem{assumption}{Assumption}

\theoremstyle{definition}
\newtheorem{example}{Example}
\newtheorem{remark}{Remark}
\newtheorem{definition}{Definition}

\usepackage{xcolor}

\title[Mirror Langevin diffusions]{Mirror Langevin diffusions: Convergence rates and Markov chain approximations}

\author{Benjamin Capdeville}
\address{ Laboratoire Math\'ematiques d'Orsay \\ Universit\'e Paris-Saclay \\ Inria ParMA \\ 91405, Orsay, France\\ {Email: benjamin.capdeville@universite-paris-saclay.fr}}
\author{Young-Heon Kim}
\address{ Department of Mathematics \\ University of British Columbia\\ Vancouver, BC, Canada\\ {Email: yhkim@math.ubc.ca}}
\author{Soumik Pal}
\address{Soumik Pal\\ Department of Mathematics \\ University of Washington\\ Seattle WA 98195, USA\\ {Email: soumik@uw.edu}}

	

\thanks{We thank Arnaud Guillin and Bo'az Klartag for helpful conversations. All authors are supported by the KARMA grant. Additionally, Pal is supported by NSF grants DMS-2502281,  DMS-2133244, DMS-2134012 and DMS-2052239.  Kim is
  partially supported by the Natural Sciences and Engineering Research
  Council of Canada (NSERC), with Discovery Grant RGPIN-2019-03926 and RGPIN-2025-06747.
Thanks to  the Pacific Institute for the Mathematical Sciences (PIMS) Kantorovich Initiative for facilitating this collaboration supported through PIMS PRN-01.
   The authors are listed in alphabetical order. 
\copyright 2026 by the
      authors. All rights reserved.}

\date{\today}

\begin{document}

\begin{abstract}
    Given a strongly convex function $u$, equip $\rr^d$ with a Riemannian metric given by the Hessian $\nabla^2 u$. This is a so-called Hessian manifold. Given a probability density $\mu$ one may run a Langevin diffusion intrinsic to the manifold with stationary distribution $\mu$. Such  (Hessian) manifold-valued Langevin diffusions are called Mirror Langevin diffusions (MLD) which have recently become popular. One of the questions we explore is whether, given $\mu$, one can choose $u$ to get an exponential convergence to equilibrium for the MLD, especially if $\mu$ is not strongly log-concave. Our results are based on Lyapunov function methods and give sufficient conditions for a Poincar\'e or a log-Sobolev inequality to hold for the MLD. These, in turn, imply exponential convergence. We also introduce a Markov chain approximation to the MLD given by a two step Gibbs sampler with stationary distribution  $\mu$. This Markov chain is a variant of the Sinkhorn Markov chain introduced in \cite{deb2023wasserstein} that is conjectured to converge to a time-inhomogeneous generalization of the MLD. Under suitable assumptions, we prove that the Markov chain has a guaranteed convergence rate in $\chi^2$ that is consistent with the diffusion time scale. Our proofs are based on ideas from entropic optimal transport and strong data processing inequalities.      
\end{abstract}

\maketitle

\section{Introduction}

Consider a probability density $\mu=e^{-V}$ on $\rr^d$. Consistent with the literature on optimal transport, we make no notational difference between a measure and its Lebesgue density. Suppose that one wants to draw a sample from $\mu$. A popular approach is to choose a diffusion with stationary distribution $\mu$ and ``fast mixing'', then choose a discrete time Markov chain approximation of the diffusion with a guaranteed convergence rate. Both these steps require choices to be made: what diffusion to pick and how to suitably discretize it? 

A common choice for the diffusion is the Langevin diffusion, a solution of the stochastic differential equation (SDE) $dX_t = - \nabla V(X_t) dt + \sqrt{2} dB_t$, where $B$ is a a multidimensional standard Brownian motion. Under suitable conditions, say, $V$ strongly convex, the diffusion is well-known to converge exponentially fast to its stationary distribution $\mu$. There is a vast and detailed literature of various discretizations of Langevin diffusions commonly used in practice, such as  Langevin MCMC \cite{CCAYBJ}, Unadjusted Langevin Algorithm \cite{DurmusMoulines, VW19}, Metropolis Adjusted Langevin Algorithm \cite{parisi81, RT96}, and Hamiltonian Monte-Carlo \cite{v018a009, pmlr-v89-mangoubi19a}. Interested readers may find a more complete set of references in the recent textbook by Chewi \cite{chewibook}.

Our focus in this paper is on a related class of diffusions called mirror Langevin diffusions (MLD) \cite{hsieh2018mirrored, zhang2020wasserstein,KC21, chewi2020exponential, deb2023wasserstein, LTVW}, which  have become popular recently due to their application in machine learning (ML) and connections to optimal transport. MLDs are Langevin diffusions over Hessian manifolds. Let us explain these concepts informally here. Formal details can be found in Section \ref{sec:hessian_prelim}. Let $u:\rr^d \rightarrow \rr$ be a smooth and strictly convex function. Equip $\rr^d$ with a Riemannian metric $g(x)=\nabla^2 u(x)$, where $\nabla^2 u(x)$ is the positive definite Hessian matrix of $u$ evaluated at the point $x$. This turns $\left(\mani=\rr^d, g\right)$ into a $d$-dimensional Riemannian manifold with a global affine coordinate chart. The measure $\mu$ remains a probability measure on $\mani$. The triplet $(\mani, g, \mu)$ is called a weighted Hessian manifold. One may consider the intrinsic manifold-valued Langevin diffusion with stationary distribution $\mu$ by replacing the gradient $\nabla$ with the Riemannian gradient $\nabla_g$ and the standard (Euclidean) Brownian motion $B$ with the Riemannian Brownian motion $B^{\mathcal{M}}$. This is the MLD. In the Euclidean coordinate chart, it satisfies the stochastic differential equation 
\begin{equation}\label{eq:sdemld_intro}
\begin{split}
dX_t 
&= - \nabla F\left( \nabla u(X_t)\right)dt + \sqrt{2 \left(\nabla^2 u(X_t)\right)^{-1}} dB_t,
\end{split}
\end{equation}
where $e^{-F}$ is the density of the push-forward of $\mu=e^{-V}$ via the map $x \mapsto \nabla u(x)$, $\nabla F$ refers to the gradient of $F$, and $B$ is a standard $d$-dimensional Brownian motion. When $u(x)=\frac{1}{2}\norm{x}^2$, the geometry becomes Euclidean, $\nabla u(x)=x$, $V=F$, and the MLD becomes the classical Langevin diffusion. 

One may run the MLD to sample from $\mu$ instead of the classical Langevin. The relevant question is whether one can choose an  appropriate $u$, depending on $V$, in such a way that there is an advantage in using the MLD over the classical Langevin diffusion.

An explicit calculation attests to this point; see Example \ref{example:explicit_benefit} for details. Fix some $\alpha \in (1,2)$. For any $\lambda >0$, let $V_\lambda(x) = \frac{1}{2}\norm{x}^2 + \lambda \norm{x}^\alpha$. The classical Langevin diffusion has exponential convergence rate $1$ in Kullback-Leibler (KL), 
 independent of $\lambda$, since this is the best constant certified by the Bakry-\'{E}mery criterion. However, if one chooses $u(x) = \norm{x}^\alpha$, the corresponding weighted Hessian manifold satisfies the curvature-dimension condition $\CD(\gamma, \infty)$ for $\gamma=\frac{\lambda \alpha}{2(\alpha-1)}$, and then $2\gamma$ becomes
the exponential rate of convergence in KL for  the corresponding MLD. But this is arbitrarily large for large $\lambda$, beating the classical Langevin by an arbitrary margin.

The earliest systematic attempt in the literature to find sufficient conditions to guarantee that $(\mani, g, \mu)$ satisfies a curvature-dimension (CD) condition is due to Kolesnikov \cite[Theorem 4.3]{kolesnikov14}. 
See also the related and contemporaneous work by Klartag \cite{klartagLogarithmicallyConcaveMomentMeasures2014a}.  The term MLD does not appear yet, being coined much later by the ML community, but these authors analyze the diffusion and its generator nonetheless. See, for example, \cite[Section 2]{kolesnikov14} and \cite[Section 4]{klartagLogarithmicallyConcaveMomentMeasures2014a}. More interestingly, these authors are inspired by an entirely different circle of ideas involving optimal transport, moments measures, the thin-shell conjecture and the dynamics of the Monge-Amp\`ere PDE. This line of analysis brings into focus the close connection of MLD to optimal transport. Think of $x\mapsto \nabla u(x)$ as a Brenier map transporting $\mu=e^{-V}$ to $\nu=e^{-F}$. Then, the three quantities $\mu, u, \nu$ are inter-related and the properties of any two of them determine the properties of the third. For example, Kolesnikov's fundamental result \cite[Theorem 4.3]{kolesnikov14} shows that if both $V$ and $F$ are convex, then the weighted Hessian manifold $(\mani, g, \mu)$ satisfies a CD$(0, \infty)$ condition. Conditions guaranteeing a CD$(\lambda, \infty)$, for some $\lambda >0$, are also available but seems  difficult to verify.

The special case of $u=V$ is worth mentioning due to its connection to sampling \cite{chewi2020exponential} as well as moment measures \cite{klartagLogarithmicallyConcaveMomentMeasures2014a}.
The MLD in this case also goes by the name Newton-Langevin diffusion. In particular, it is shown in \cite[Theorem 1]{chewi2020exponential} that when $u=V$, in \textit{any} dimension, the MLD converges to $\mu$ in the $\chi^2$ divergence exponentially at rate $1$. This is proved by an application of the Brascamp-Lieb inequality that does not extend beyond this special case. Alternatively, it also follows from \cite[Theorem 4.3]{kolesnikov14}. 

We take the route of Lyapunov function based methods pioneered in a sequence of articles by Bakry, Cattiaux, Gullin, Wang and coauthors  \cite{Bakry2007RateOC, CGWW, CG10, CG16, BPCG08}. The advantage of this method is that it is easy to verify and suits our question of designing a $u$ for a given $V$. The disadvantage is that the constants do not explicitly show their dimension dependence. One of our results, Theorem \ref{prop:PI}, shows roughly that if we assume that
\begin{equation}\label{eq:intro-lyapunov}
\lim_{x\rightarrow \infty}\left[  x \cdot \frac{\partial V}{\partial x} - \alpha x^T \nabla^2u(x) x\right]=\infty,
\end{equation}
in addition to some smoothness of $u$ and boundedness on its Hessian, then the MLD satisfies a Poincar\'e inequality and converges exponentially fast. A similar condition is derived in Theorem \ref{prop:LSI} for the log-Sobolev inequality. 

 A benefit our result is that neither $\mu$ nor $\nu$ is required to be strongly log-concave. Condition~\eqref{eq:intro-lyapunov} is a
statement about the interplay between the growth of $V$ and the growth of the
metric $\nabla^2u$ at infinity. Moreover, the condition involves only the first derivatives of $V$ and the quadratic form of $\nabla^2 u$, in contrast with the curvature-dimension
condition $\CD(\lambda,\infty)$, whose verification on a Hessian manifold requires controlling the Ricci curvature~\eqref{eq:ricci}, which is an
expression in the third derivatives of $u$.

Our main set of results (in Section \ref{sec:gibbsmc}) is about a Markov chain approximation to the MLD with guaranteed convergence rate that is consistent with the diffusion time-scaling. More precisely, for every $\eps >0$, we construct a Markov chain $\left( X_k^\eps,\; k =0,1,2,\ldots \right)$, whose stationary distribution is \textit{exactly} $\mu$, such that the following two results hold. 

Under Assumption \ref{asmp:MCconv} that roughly says $u$ and its convex conjugate $u^*$ are uniformly  convex and have bounded derivatives up to sixth order, and $V$ has bounded second derivatives, Theorem \ref{thm:mcconv} shows that the continuous-time interpolated process $\left( X^\eps_{\lfloor t/\eps \rfloor},\; t\ge 0 \right)$ converges weakly in the Skorokhod topology to the MLD, as $\eps \downarrow 0$. Additionally, if $\nu=e^{-F}$ satisfies a classical Poincar\'e inequality and $\nabla F$ is  $L$-Lipschitz, then Theorem \ref{thm:mclya} states that the Markov chain  $\left(X_k^\eps,\; k=0,1,2,\ldots\right)$ converges exponentially fast as $k\to \infty$, for every $\eps>0$ in chi-square divergence. 
That is, if $p_k^\eps$ is the density of $X_k^\eps$, starting with an initial density $p_0^\eps$, then, for some positive constant $c_0$, for all $\eps \in (0,1)$, 
\begin{equation}\label{eq:chisqintro}
\chi^2\big(p^{\epsilon}_k \mid \mu\big)
\;\le\; (1 - c_0 \epsilon)^{k}\,
\chi^2\big(p^{\epsilon}_0 \mid \mu\big),
\qquad k \in \mathbb{N}.
\end{equation}
Here, $\chi^2$ refers to the chi-square divergence between two probability measures. Since $\KL$ is dominated by $\chi^2$ divergence, the same exponential decay holds for $\KL$ as well. These two theorems are consistent with each other in the following sense: Suppose that the limiting MLD mixes in $O(1)$ in continuous time. Since time is rescaled in step size $\epsilon$ in discrete time, one would expect the Markov chain to mix in $O(1/\epsilon)$ many steps. This is consistent with the exponential convergence in \eqref{eq:chisqintro}. Our proof follows by a comparison of discrete Dirichlet energies and a recent result on Strong Data Processing Inequality (SDPI) by Klartag and Ordentlich \cite{KlartagOrdentlich}.  We expect the proof technique to be useful for other Gibbs samplers and the convergence rate of low-temperature Sinkhorn iterations as well; see e.g. \cite{CCGT25}.

We now describe the construction of this Markov chain and its connection to the Schr\"odinger bridge and the Sinkhorn algorithm that inspired it. 

\subsection{The two-step Gibbs sampler and its connection to Sinkhorn}
\label{sec:intro-chain}

The Markov chain is a two-step Gibbs sampler. Define the
conditional density $q_\epsilon(y \mid x)$ to be the density of $N\Big(\nabla u(x),\; \epsilon\, \nabla^2 u(x)\Big)$.
Consider the joint density on
$\mathbb{R}^d \times \mathbb{R}^d$ given by
\[
\pi_\eps(x,y) = e^{-V(x)}\, q_\epsilon(y \mid x).
\]
This joint density has two conditional densities $q_\eps(y \mid x)$, obviously, and $\hat{q}_\eps(x\mid y)$, defined via Bayes' formula. Given $X_0=x$, one can generate a Markov kernel by first sampling an auxiliary variable $Y$ from $q_\eps(\cdot \mid x)$ and then, given $Y=y$, sampling $X_1$ from $\hat{q}_\eps(\cdot \mid y)$. 

Such a construction is a particular example of a two stage (or two component) Gibbs sampler. Suppose $\eta(x,y)$ is some joint density on $\rr^d \times \rr^m$, under which the $x$-coordinate is distributed according to $\mu$. Let $K(y\mid x)$ denote the conditional density of $Y$ at $y$, given $X=x$, and let $\hat{K}(x\mid y)$ denote the conditional density of $X$ at $x$, given $Y=y$. Then, a Gibbs sampler or Glauber dynamics will successively sample $Y_n$, given $X_n=x$, from the density $K(\cdot \mid x)$, and then, given $Y_n=y$, sample $X_{n+1}$ from the density $\hat{K}(\cdot \mid y)$, and so on. It is well known that, under mild conditions, the Markov chain $(X_n, Y_n)$, $n\ge 1$, converges to the stationary distribution $\eta$. However, the embedded chain $(X_n,\; n\ge 1)$ is also a (reversible) Markov chain with kernel $\int \hat{K}(\cdot \mid y)K(y \mid x)dy$ \cite[Definition 4.4]{Raginsky2014StrongDP} that converges to $\mu$. This fact is often used in sampling with incomplete data \cite{LiuWongKong} and in sampling problems arising in Bayesian statistics \cite{DiaKSC}.
Our proposed Markov chain is the special case when $\eta=\pi_\eps$. The significance of this particular choice is underlined by our diffusion approximation. 

The connection to the Sinkhorn algorithm is the following. Consider the entropy-regularized optimal transport problem \cite{schroLeonard13} between the marginals $\mu=e^{-V}$ and $\nu=e^{-F}= (\nabla u)_{\# \mu}$ described below \eqref{eq:sdemld_intro}. This is the solution of the following one-parameter family of optimization problems
\[
\argmin_{\gamma\in \Pi(\mu, \nu)}\left[ \frac{1}{2}\int \norm{y-x}^2 d\gamma + \eps\mathrm{Ent}(\gamma)\right],\quad \eps >0.
\]
Here $\Pi(\mu, \nu)$ refers to the set of joint distributions (may assume densities in this case) with marginals $\mu$ and $\nu$, and for any such joint density $\gamma$, its entropy is defined as $\mathrm{Ent}(\gamma)=\int \gamma(x,y)\log \gamma(x,y)dxdy$. For $\eps=0$, the above problem is simply the Monge-Kantorovich optimal transport problem. For $\eps >0$ there is a unique solution to this problem which is called the Schr\"odinger bridge at temperature $\eps$. This joint density may be solved via a dynamic algorithm called the Sinkhorn or the IPFP algorithm.

Suppose $\eta_\eps(x,y)$ is this Schr\"odinger bridge at temperature $\eps$ between $\mu$ and $\nu$. If we run the two stage Gibbs sampler with $\eta_\eps$ then the embedded Markov chain on the X-marginal $(X_n,\; n\ge 1)$ is the stationary version of the Sinkhorn Markov chain described in \cite[Section 4]{deb2023wasserstein}. The Sinkhorn Markov chain is conjectured (see the discussion above \cite[Theorem 4.4]{deb2023wasserstein}) to converge to a time-inhomogeneous generalization of the MLD, called the Sinkhorn diffusion, where the Hessian manifold itself changes with time, being induced by a time-varying family of convex potentials $(u_t,\; t\ge 0)$ satisfying the parabolic Monge-Amp\`ere PDE. In the stationary case when $u_t\equiv u$, for al $t\ge 0$, this process is just the MLD that we consider in this paper. See  below \cite[Theorem 1.4]{deb2023wasserstein}. 

In practice, however, the Schr\"odinger bridge is not explicitly known in most cases. However, one would expect that as a universal continuum limit, the Sinkhorn diffusion should be a diffusion limit for other Markov chains that are similar to the Sinkhorn chain. Our proposed Markov chain replaces the Schr\"odinger bridge by its Gaussian approximation from \cite{pal2019difference} and we show that this replacement is sufficient for approximating the MLD.   


In the Euclidean case $u(x) = \frac{1}{2}\norm{x}^2$ the construction specializes
to $q_\epsilon(\cdot \mid x) = N(x, \epsilon I)$, and the chain coincides
with the recent construction of proximal sampler of \cite{LeeShenTian}, whose convergence under
Poincar\'e and log-Sobolev inequalities was established in
\cite{ChenChewiSalimWibisono} by related entropic-contraction arguments. In fact, in Section 4.3.3 in \cite{ChenChewiSalimWibisono} the authors note the connection of the proximal sampler chain to the JKO method to discretize Wasserstein gradient flows where one replaces the Wasserstein metric with its entropy-regularized counterpart. Since the MLD is also the Wasserstein gradient flow of entropy in the Hessian manifold \cite{lisini2009nonlinear}, 
our chain may therefore be viewed as the extension of the proximal sampler to Hessian geometries where the isotropic Gaussian  $N(x, \epsilon I)$ is replaced with $N(x^*, \epsilon\nabla^2 u(x))$ that is adapted to the geometry. Conversely, the connection to the Schr\"odinger bridge described above supplies another optimal-transport interpretation of the proximal sampler itself that we believe is new.


 Practical implementation of our Markov chain requires us to be able to sample from the reverse conditional density $\hat{q}_\eps(x\mid y)$. This can be a genuine problem. Our hope is that, for small $\eps$, a strategy similar to \cite{LeeShenTian} can help. See the discussion below Definition 1 in \cite{LeeShenTian}.   

Lastly, we should mention that several other discretization methods have been studied for the MLD in recent papers such as \cite{zhang2020wasserstein, KC21, LTVW}. These methods discretize time in the step size $\eps$. In contrast to our Markov chain, their iterations are no longer stationary with respect to $\mu$ for arbitrary $\eps >0$. The aim  then is to show that the \textit{bias} goes to zero as $\eps \downarrow 0$. In \cite{zhang2020wasserstein} the authors combine the usual Euler Maruyama scheme with a mirror descent step to study a discretization they call Hessian Riemannian Langevin Monte Carlo (HRLMC). They showed that under certain assumptions, the iterates converge into a Wasserstein ball around $\mu$. It was shown in \cite{LTVW} that this scheme indeed converges with vanishing bias, i.e., the radius of the Wasserstein ball converges to zero with the step size in the HRLMC scheme. A different discretization scheme called the Mirror Langevin Algorithm (MLA) was proposed in \cite{KC21} which also leads to vanishing bias as the step-size goes to zero. The MLA attempts to do a finer discretization of the MLD than the Euler Maruyama scheme by using multiple inner iterations. However, for $\eps >0$ it still retains a bias, while our Markov chain is unbiased.

\section{Preliminaries }

\subsection{Hessian manifolds and dual affine coordinates}\label{sec:hessian_prelim}

Let $u:\rr^d \rightarrow \rr$ denote a strictly convex function. We will later assume further smoothness conditions on $u$. For now, it suffices to assume that $u \in C^4$, i.e., four times continuously differentiable. In particular, the map $x \mapsto \nabla u(x)$ is a diffeomorphism from $\rr^d$ to itself. 

Throughout this article let $\left(\mani, g \right)$ denote a Hessian manifold \cite{shima2007geometry} generated $u$. This means that we equip $\rr^d$ with a Riemannian metric tensor $g(x)=\nabla^2 u(x)$, where $\nabla^2 u(x)$ is the positive definite Hessian matrix of $u$ evaluated at the point $x$. This turns $\left(\mani=\rr^d, g\right)$ to a $d$-dimensional manifold with a global affine coordinate chart. Hence, we will not notationally distinguish between a point $x\in \rr^d$ and its image on $\mani$. 

In the fixed affine coordinate system we write the components functions of the metric $g$ as the matrix $(g_{ij}(x))$. The inverse, or the co-metric, $g^{-1}$, appears in many calculations, we write its components with upper indices $(g^{ij}(x))$. Let $d: \mani \times \mani \to [0,+\infty)$ denote the distance on $\mani$ induced by the metric. For $x \in \mani$, $T_{x}\mani$ denotes the tangent space of $\mani$ at $x$, $T\mani$ denotes the tangent bundle, and $\Ric_{g}(x): T_{x}\mani \times T_{x}\mani \to \mathbb{R}$ denotes the Ricci curvature tensor at $x$. The scalar product and the associated norm on the tangent space $T_x \mani$ will be denoted by $\iprod{\cdot,\cdot}_g$ and $\norm{\cdot}_g$.

We use the Einstein summation notation throughout this paper. As an example, let $(\Gamma_{ij}^{k}(x))$ denote the Christoffel symbols for the Levi-Civita connection. It can be easily shown that, in the affine coordinate chart $(x^1,\dots,x^d)$, $\Gamma_{ij}^{k}= \frac{1}{2}g^{kl} u_{ijl}$.
According to the Einstein summation notation, the RHS is a shorthand for $\frac{1}{2} \sum_{l=1}^d g^{kl} u_{ijl}$. 

Given a smooth function $\phi :\mani \rightarrow \rr$, its Riemannian gradient and Hessian at a point $x$ will be denoted by $\nabla_g \phi(x)$ and $\Hess_g \phi(x)$. In coordinate notations, it follows that 
\[
\begin{split}
\nabla_g \phi(x) &= \left(\nabla^2 u(x) \right)^{-1}\nabla \phi(x),\\
\left(\Hess_g \phi(x)\right)_{ij} &= \partial_{ij} \phi(x) - \Gamma^k_{ij} \partial_k \phi.
\end{split}
\]
where $\nabla$, $\partial_k$ and $\partial_{ij}$ refer to the usual gradient, partial derivative and second partial derivatives with respect to the affine coordinates. 

The volume measure is a measure on the Borel sigma-algebra of $\mani$ which, in the affine coordinate system, admits the following density with respect to the $d$-dimensional Lebesgue measure
\begin{equation}\label{eq:vol-measure-coords}
\vol(dx) = \sqrt{\det g(x)}dx. 
\end{equation}
For any other measure $\mu$ on $\mani$ we will assume that $\mu$ is absolutely continuous with respect to $\vol$, i.e., $\mu(dx)= e^{-\phi(x)}\vol(dx)$, for some function $\phi: \mani \rightarrow \rr \cup\{+\infty\}$.  

A Hessian manifold comes with a special dual affine coordinate chart. This dual symmetry is often very useful. Define $x^*=\nabla u(x)$. For a given point $p \in \mani$, the original affine coordinate $x\in \rr^d$ will be called its primal coordinate. On the other hand, $x^*\in \rr^d$ will be called its dual coordinate representation. Similarly, given a dual coordinate $y$ of a point $p \in \mani$, $y_*=(\nabla u)^{-1}(y)$ represents its primal coordinate. By our assumption on $u$ the map $x\mapsto x^*$ is a diffeomorphism of $\rr^d$. 

The importance of writing in this way is that there is a natural duality under which these two coordinate charts behave symmetrically by identifying the Hessian manifolds generated by $u$ and its Legendre-Fenchel convex dual $u^*$ \cite[Proposition 2.7]{shima2007geometry}. The conversion between the two coordinate charts is via the Jacobian $\frac{\partial x^*}{\partial x}=g(x)$. For example, given a differentiable function $\varphi: \rr^d \rightarrow \rr$, one may write
\[
\frac{\partial \varphi}{\partial x^*}(x):= \frac{\partial x}{\partial x^*} \frac{\partial\varphi}{\partial x}(x)= g^{-1}(x) \frac{\partial\varphi}{\partial x}(x)=\nabla_g \phi(x).
\]
In particular, $\frac{\partial \phi}{\partial x^*}$ is the Riemannian gradient of $\phi$ at $x$, an intrinsic, i.e, coordinate invariant, quantity. 

The quantity $G(x):=\frac{1}{2}\log \det g(x)$ makes a recurrent appearance in this paper along with its gradient and Hessian. By a straightforward computation
\begin{equation}\label{eq:gradG}
\begin{split}
\left(\nabla_g G(x)\right)_k &= \frac{1}{2} g^{kl} \partial_l \log \det g(x)= \frac{1}{2} g^{kl} g^{ij} u_{ijl}, \quad u_{ijl}=\frac{\partial^3u}{\partial x_i \partial x_j\partial x_l}.
\end{split}
\end{equation}
That is, $\nabla_g G(x)=\frac{1}{2} g^{-1}\Tr(g^{-1} \nabla g)(x)$, where $\Tr(g^{-1} \nabla g)$ should be interpreted in the tensor notation. 

Divergence and Laplacian operators will be denoted by $\div_g$ and $\Delta_g$, respectively. In the primal affine coordinates, $\div_g$ has the expression 
\[
\div_g(X)= \partial_i X^i + X^i \partial_i G(x),
\]
where $X$ is a smooth vector field. Finally, $\Delta_g:=\div_g \nabla_g$. 

Finally we need to compute $\Hess_g G$ and the Ricci curvature $\Ric_g$. Recall that $\Gamma_{ij}^k = \frac{1}{2} g^{kl} u_{ijl}$. Since $\partial_m g^{kl} = - g^{kr} u_{rsm} g^{sl}$, 
\[
\partial_m \Gamma_{ij}^k =  - \frac{1}{2}  g^{kr} u_{rsm} g^{sl}u_{ijl} + \frac{1}{2} g^{kl} u_{ijlm}.
\]
Hence, a quick computation (see \cite[Section 3]{kolesnikov14}) shows that
\begin{equation}\label{eq:ricci}
\begin{split}
    \left( \Ric_g \right)_{ij} =& \, \partial_k \Gamma_{ij}^k - \partial_j \Gamma_{ik}^k + \Gamma_{ij}^k \Gamma_{ks}^s - \Gamma_{is}^k \Gamma_{jk}^s \\
    =& \,  \frac{1}{4} g^{kr} g^{sl} \left( u_{ikl} u_{jrs}  - u_{ijl} u_{krs}   \right)
\end{split}    
\end{equation}

\begin{remark}\label{rmk:riccibnd}
    We will later need condition that the Ricci curvature is uniformly bounded below, i.e., there exists a positive constant $c_0$ such that $\Ric_g(x) \ge c_0 g(x)$ for all $x\in \rr^d$. This is true, for example, if there are positive constants $\alpha_0, \alpha_1, \alpha_2$ such that
\[
\alpha_0 I \le g(x) \le \alpha_1 I, \quad \sup_x\sup_{ijk} \abs{u_{ijk}(x)} \le \alpha_2.
\]
\end{remark}


\subsection{Mirror Langevin diffusions as Langevin diffusions on Hessian manifolds}

A standard textbook for Brownian motion and stochastic calculus on manifolds is \cite{hsu-stoch-analysis-manifold} where the reader can find more details. 

A process $\left(B^\mani_t, t \geq 0\right)$ on $\mani$, given some initial distribution $\mu$, is called a (manifold) Brownian motion if it is a $\frac{1}{2}\Delta_g$-diffusion process \cite[Proposition 3.2.1]{hsu-stoch-analysis-manifold}. Here, $\Delta_g$ is the Laplace-Beltrami operator on $(\mani, g)$, expressed in local coordinates for $u \in C_c^{\infty}(\mani)$ as
\begin{align}\label{eq:laplace-beltrami-coords}
    \Delta_g u &= \frac{1}{\sqrt{\det g}}\frac{\partial}{\partial x^i} \left(\left(\sqrt{\det g}\right) g^{ij} \frac{\partial}{\partial x^j} u\right) = g^{ij}\frac{\partial^2}{\partial x^i\partial x^j}u-g^{ij}\Gamma_{ij}^{k}\frac{\partial}{\partial x^k}u.
\end{align}

When the Ricci curvature of $M$ has a global constant lower bound (see Remark \ref{rmk:riccibnd}), the manifold Brownian motion exists and has a.s.\ infinite explosion time \cite[Theorem 3.5.3]{hsu-stoch-analysis-manifold}. 

In a local coordinate system, the manifold Brownian motion is a weak solution to the following SDE 
\cite[equation (3.3.11)]{hsu-stoch-analysis-manifold}
\begin{align}\label{eq:local-mani-bm}
    dB^\mani_t = \sqrt{g^{-1}(B^\mani_t)}dB_t - \frac{1}{2}g^{ij}(B^\mani_t)\Gamma_{ij}^{k}(B^\mani_t)dt,
\end{align}
where $(B_t, t \geq 0)$ is a standard $d$-dimensional Euclidean Brownian motion. The manifold Brownian motion is a reversible process with stationary measure equal to $\vol$. For the special case of Hessian manifolds, by substituting the expression for the Christoffel symbols, 
\begin{equation}\label{eq:local-Hessian-bm}
\begin{split}
    dB^\mani_t &= \sqrt{g^{-1}(B^\mani_t)}dW_t - \frac{1}{4}g^{kl}(B^\mani_t)g^{ij}(B^\mani_t)u_{ijl}(B^\mani_t)dt,\\
    &= \sqrt{g^{-1}(B^\mani_t)}dW_t - \frac{1}{2}\nabla_g G\left((B^\mani_t\right)dt,
\end{split}    
\end{equation}
where $G(x)=\frac{1}{2}\log \det g(x)$, as defined above.

Let $(B_t^{\mani}, t \geq 0)$ denote the Brownian motion on $M$, and let $U \in C^2(M)$ be a suitable potential such that $\mu=e^{-U}d\vol$ is a probability measure on $\mani$. Consider a weak solution to the following manifold-valued SDE
\begin{equation}\label{eq:reference-process-intro}
    dX_t = -\nabla_g U(X_t) dt +  dB_{2t}^{\mani}, 
\end{equation}
where $B_{2\cdot}^\mani$ is the manifold Brownian motion running at twice the speed. It can be shows that the above process (under suitable assumptions) is reversible with stationary measure given by $\mu$. We call $X$ to be the Manifold Langevin diffusion with stationary measure $\mu$. 

Note that the Markov generator of the above process is given by 
\begin{equation}\label{eq:refprocgen}
L := - \iprod{\nabla_g U, \nabla_g}_g + \Delta_g,
\end{equation}
where $\iprod{\cdot , \cdot  }_g$ is the Riemannian inner product of the metric $g$.

Now suppose $\mu(x)=e^{-V(x)}$ is a probability density function on $\rr^d$. It induces a probability measure on $\mani$. By an abuse of notation we retain the notation $\mu$ for this measure as well. By a further abuse of notation we will denote a probability measure and its associated density by the same notation. The context will make it clear. 

If $\mu$ admits a density $e^{-\phi}$ with respect to the volume measure on $\mani$, it follows that 
\begin{equation}\label{eq:density-manifold}
    \phi(x)= V(x) + \frac{1}{2}\log \det g(x)= V(x) + G(x).
\end{equation}

Let us write down the SDE for the manifold Langevin diffusion \eqref{eq:reference-process-intro} in the affine coordinate system. See \cite[eqn. (21)]{MP25}.
\[
\begin{split}
dX_t &= - \nabla_g \phi(X_t) dt +  d B_{2t}^\mani\\
&= - \nabla_g\left[ V(X_t) + 2 G(X_t)\right]dt +  \sqrt{2 g^{-1}\left( X_t\right)} dB_t,
\end{split}
\]
where $B$ is a standard $d$-dimensional Euclidean Brownian motion.

Let $\nu=e^{-F}$ denote the density that the pushforward of the density $\mu=e^{-V}$ by the map $x\mapsto \nabla x^*=u(x)$. That is, 
\begin{equation}\label{eq:jacobian}
F(x^*) = V(x) + \log \det\left( \nabla^2 u(x)\right)= V(x) + 2G(x).  
\end{equation}
Thus, one may write the SDE for $X$ as 
\begin{equation}\label{eq:sdemld}
\begin{split}
dX_t &= - g^{-1}(X_t) \nabla F(\nabla u(X_t)) dt + \sqrt{ 2g^{-1}(X_t)}dB_t\\
&= - \frac{\partial F}{\partial x^*}\left( X_t^* \right)dt + \sqrt{2 \left(\nabla^2 u(X_t)\right)^{-1}} dB_t. 
\end{split}\tag{Primal}
\end{equation}
This is the SDE for the Mirror Langevin diffusion introduced in \cite{zhang2020wasserstein}. The process is stationary with respect to the density $\mu=e^{-V}$. We are going to assume throughout that a unique weak solution to \eqref{eq:sdemld} exists. A sufficient condition, see \cite[Proposition 4]{klartagLogarithmicallyConcaveMomentMeasures2014a}, are that $\mu, \nu$ are fully supported, in $C^1$, and
\begin{equation}\label{eq:wkexistence}
\inf_{y \in \rr^d} y \cdot \nabla F(y) > -\infty.  
\end{equation}

But, just like $(\mani, g)$ has two coordinate charts, primal and dual, that are equivalent, there is also an equivalent SDE representation of the MLD in the dual coordinates. Let $Y_t=X_t^*$, i.e., $X_t=(Y_t)_*$, $t\ge 0$. Then, the process $Y$ is also an MLD with a stationary density $\nu=e^{-F}$ and satisfying the SDE
\begin{equation}\label{eq:sdemlddual}
\begin{split}
dY_t &= - \frac{\partial V}{\partial y_*}\left( (Y_t)_* \right)dt + \sqrt{2 \left( \nabla^2 u^*(Y_t) \right)^{-1} } dB_t,
\end{split}\tag{Dual}
\end{equation}
where $B$ is a standard Brownian motion, and $\nabla^2 u^*$ is the metric tensor in the dual coordinate chart. See \cite[Theorem 3.5]{deb2023wasserstein} for a proof in a more general situation. The key insight for the next section is that, although the primal and the dual SDEs in Euclidean coordinates refer to two different diffusions with different stationary measures, they are in fact the same Langevin diffusion on the Hessian manifold written in two equivalent coordinate charts. An immediate application of this idea is the sufficient condition \eqref{eq:wkexistence} for the weak existence of the primal MLD \eqref{eq:sdemld} may be replaced by
\begin{equation}\label{eq:wkexistence2}
\inf_{x \in \rr^d} x \cdot \nabla V(x) > -\infty.  
\end{equation}
This is because, by \cite[Proposition 4]{klartagLogarithmicallyConcaveMomentMeasures2014a}, \eqref{eq:wkexistence2} is a sufficient condition for the existence of the Dual MLD, which, in turn, implies the existence of the primal MLD, and vice versa.

When $u(x)=\frac{1}{2}\norm{x}^2$, i.e., $g\equiv I$, Both the primal and the dual SDEs reduces to the SDE for the classical Langevin diffusion.

\subsection{MLD as the Wasserstein gradient flow of Kullback-Leibler on Hessian manifolds}

Let $(\mani, g)$ be a Hessian Riemannian manifold defined above. One can define the Wasserstein space $\wass_g$ of all Borel integrable probability measures on $\mani$ with finite second moments, equipped with the Wasserstein-$2$ distance with respect to the underlying Riemannian distance induced by $g$. 

Assume that the measure $\mu$ is in $\wass_g$. For any other density $\rho$ on $\rr^d$, define the KL divergence (i.e., relative entropy) as 
\[
\KL(\rho \mid \mu):= \int \rho(x) \log \frac{\rho(x)}{\mu(x)} dx. 
\]
For probability measures $\rho$ that are not absolutely continuous, let $\KL(\rho \mid \mu)=\infty$. The definition does not change if we consider the density with respect to the volume measure on $\mani$. Hence, the function $\KL(\cdot \mid \mu)$ may be considered as a function on $\wass_g$ and, as such, one may ask if it admits a Wasserstein gradient flow. This is made precise in Lisini \cite{lisini2009nonlinear} who describes the infinite-dimensional Riemannian structure of $\wass_g$ and identifies the PDEs satisfied by natural gradient flows on $\wass_g$. See \cite[Section 1]{lisini2009nonlinear}.



In fact, in \cite[Theorem 1.1]{lisini2009nonlinear}, it is shown that, when (i) $V$ is convex, lower semicontinuous and bounded from below, and (ii) $u\in C^3$ and $g$ is bounded above and below by positive constants times identity, then the marginal flow of the MLD can be rigorously shown as the gradient flow of KL as a limit of a suitable JKO-type minimizing movement scheme on $\mani$.


\section{The rate of convergence of MLD}

Consider the question of convergence rate to equilibrium for the MLD. Since the MLD is the Langevin diffusion on the Hessian manifold, a natural condition under which it has exponential convergence is the so-called curvature-dimension (CD) condition. Although this is difficult to verify in practice, and our primary contribution in this paper is the alternative method of Lyapunov functions, we describe below the CD condition for comparison. See the Appendix for a fuller analysis of CD in dimension one. 

Assume that the Hessian manifold $(\mani, g)$ is geodesically complete. Let $\mu=e^{-\phi}d\vol$ denote a probability measure on $(\mani, g)$. Call the triplet $\left( \mani, g, \mu\right)$ as a weighted Riemannian manifold according to the terminology in \cite{BakryGentilLedoux2014}. Recall the Curvature-Dimension condition $\CD(\lambda, \infty)$ \cite[Chapter 1.16.2]{BakryGentilLedoux2014} for a weighted Riemannian manifold. Note that our measure $\mu$ has a density $e^{-\phi}$ with respect to the volume measure but $e^{-V}$ with respect to the Lebesgue measure on the primal space (and $e^{-F}$ w.r.t. the Lebesgue measure on the dual space).

\begin{definition}\label{defn:curvature-dimension}
Say that the weighted Riemannian manifold $(\mani, g, \mu)$ satisfies the $\CD(\lambda, \infty)$ condition, for $\lambda \ge 0$, if 
\begin{equation}\label{eq:curvature-dimension}
\Ric_g(x) + \Hess_g(\phi)(x) \succcurlyeq \lambda g(x), \quad \forall\; x \in \rr^d, 
\end{equation}
where the inequality is in the sense of p.s.d. matrices. 
\end{definition}

Consider the MLD as the Langevin diffusion on the Hessian manifold with generator \eqref{eq:refprocgen}. It follows from \cite[Proposition 5.7.1]{BakryGentilLedoux2014} that if $(\mani, g, \mu)$ satisfies the $\CD(\lambda, \infty)$ condition then (i) the measure $\mu$ satisfies a logarithmic-Sobolev inequality $\LSI(1/\lambda)$, and (ii) the MLD converges exponentially fast to equilibrium in $\KL$ at rate $2\lambda$. 

As mentioned in the Introduction, a fundamental result due to Kolesnikov \cite[Corollary 4.2, Theorem 4.3]{kolesnikov14} shows that if both $\mu$ and $(\nabla u)_{\#\mu}$ are log-concave, then $(\mani, g, \mu)$ satisfies CD$(0, \infty)$. Conditions for a positive $\lambda$ are harder to veirfy except for special cases such as $u=V$. However, the following explicit computation clearly demonstrates the benefits of using MLD over classical Langevin. 

\begin{example}\label{example:explicit_benefit}
Take the dimension $d\ge 2$. Fix $\alpha \in (1, 2)$ and $\lambda >0$. Consider a potential 
    \[
      V_\lambda(x) = \frac{1}{2}\norm{x}^2 + \lambda \norm{x}^\alpha +C, 
    \]
    where $C$ is the constant such that $\mu= e^{-V_\lambda}$ is a probability density. 
One can run the Langevin diffusion. Since $V_\lambda \ge \frac{1}{2} \norm{x}^2$, one obtains an exponential rate of convergence of one in KL, irrespective of $\lambda$. On the other hand, one can choose $u(x)= \norm{x}^\alpha$ and run a mirror Langevin diffusion with stationary distribution $\mu$. For this choice of $u$ let us verify the CD condition.

Consider $x \neq 0$. Let $r=\norm{x}$, for simplicity. Then, 
\[
g(x)=\nabla^2u(x)= \frac{\alpha}{ r^{2-\alpha}}\left( I +  (\alpha -2)  \frac{x x^T}{r^2}\right).
\]
By the  Sherman-Morrison formula, $\log \det g(x)=\log(\alpha-1) + d \log \alpha + d(\alpha -2) \log r$. Note that, this choice of $u$ is not admissible, since $g, g^{-1}$ do not exist at the origin. This can be easily fixed by either, convoluting with a positive mollifier (that preserved convexity) or, by considering a smoothed out version such as $u(x)=\left( \delta + \norm{x}^2 \right)^{\alpha /2}$, for some positive $\delta \approx 0$. We are going to ignore this mollification below and assume that the Hessian remains bounded in a neighborhood of zero.

After some computations, we get from \eqref{eq:ricci},  
\[
\Ric_g(x)=-\frac{(d-2)(\alpha -2)^2}{4(\alpha-1)} \frac{1}{r^2}\left(  I -  \frac{x x^T}{r^2}\right).
\]
Similarly, a tedious but straightforward calculation shows that $\nabla^2 V_\lambda= I + \lambda g$ and 
    $\Hess_g(V_\lambda)= I + C(\alpha, \lambda, \norm{x}) g$, where 
    \[
    \begin{split}
       C(\alpha, \lambda, \norm{x})&= \frac{\lambda \alpha}{2(\alpha -1)} + \frac{2-\alpha}{2\alpha (\alpha -1)}\norm{x}^{2-\alpha}.
    \end{split}
    \]
    In particular, 
    $
        \Hess_g(V_\lambda)\succcurlyeq \frac{\lambda \alpha}{2(\alpha -1)}  g.
    $
   
We now look for a bound of the form $\Ric_g+\Hess_g(V_\lambda) \succcurlyeq K g$, by adding the two terms above. Writing $P = xx^T/r^2$ and $P_\perp = I-P$,
\[
g(x) = \alpha r^{\alpha-2}\bigl[(\alpha-1) P + P_\perp\bigr], \qquad
\Ric_g(x) = -\frac{(d-2)(\alpha-2)^2}{4(\alpha-1)}\,\frac{1}{r^2}\,P_\perp,
\]
Thus, 
\[
\Ric_g(x) + \Hess_g(V_\lambda) \succcurlyeq  -\frac{(d-2)(\alpha-2)^2}{4(\alpha-1)}\,\frac{1}{r^2}P_\perp + \frac{\lambda \alpha}{2(\alpha -1)}g.  
\]

When the dimension $d=2$, the Ricci term vanishes identically, so $\Ric_g+\Hess_g(V_\lambda) = \Hess_g(V_\lambda) \succcurlyeq \frac{\lambda\alpha}{2(\alpha-1)}\,g$ holds for every $r>0$.

Now let $d>2$. Due to our mollifictaion $u(x)=\left( \delta + \norm{x}^2 \right)^{\alpha /2}$ around the origin, both $g$ and $\Ric_{g}$ stay bounded near the origin. On the other hand, as $r\rightarrow \infty$, the Ricci curvature decays like $-r^{-2}$ while $g$ grows like $r^{\alpha-2}$.  Consequently, for every fixed $\delta>0$, one can find $M_\delta = M_\delta(\alpha,d)>0$ irrespective of $\lambda$, such that 
\[
\Ric_{g} + \Hess_{g}(V_\lambda) \;\succcurlyeq\; \Bigl[\tfrac{\lambda \alpha}{2(\alpha-1)} - \,M_\delta \Bigr] g \;=:\; K(\lambda,\delta)\, g.
\]
Since $M_\delta$ is independent of $\lambda$, $K(\lambda,\delta)\to\infty$ as $\lambda\to\infty$ for any fixed $\delta$.

Either way, for any $d\ge 2$, for all $\lambda$ large enough the constant $K(\lambda,\delta)$ can be made as large as we wish. Thus, for all such large $\lambda$, one has a guaranteed faster rate of convergence for the MLD over the usual Langevin diffusion.
\end{example}

\subsection{The method of Lyapunov functions} 
We are going to assume throughout that there is a weak solution of the MLD (see \eqref{eq:wkexistence} and \eqref{eq:wkexistence2}). In particular, $\mu$ and $\nu$ are fully supported on $\rr^d$.

Let $\genmld$ denote the extended generator of the  SDE \eqref{eq:sdemld} acting on $C^2(\rr^d)$ functions. For $\xi \in C^2$, 
\[
\genmld \xi(x)= - \frac{\partial \xi}{\partial x}(x) \cdot \frac{\partial F}{\partial x^*}(x^*) + \Tr\left( \frac{\partial x}{\partial x^*} \nabla_x^2 \xi(x) \right)
\]
Recall that, according to our notations,  $g(x)=\nabla^2 u(x)=\frac{\partial x^*}{\partial x}$,  $\nabla_g = \frac{\partial x}{\partial x^*} \frac{\partial}{\partial x}$. A quick computation shows (see  \cite[Section 3]{deb2023wasserstein}) that for suitable smooth functions $\xi:\rr^d\rightarrow \rr$ the Dirichlet energy function for this diffusion is given by 
\[
\cdc_{P}(\xi):=\int \left(\nabla_x \xi\right)^T \left( \nabla^2 u(x)\right)^{-1} \left(\nabla_x \xi\right) d\mu(x).
\]
Similarly for the dual MLD \eqref{eq:sdemlddual}, the corresponding Dirichlet energy function is given by 
\[
\cdc_{D}(\xi):=\int \left(\nabla_y \xi\right)^T \left( \nabla^2 u^*(y)\right)^{-1} \left(\nabla_y \xi\right) d\nu(y).
\]
But it is much more convenient to think of the generator (see \eqref{eq:refprocgen}) and the Dirichlet energy $\cdc$ on functions on the Hessian manifold where they assume a more natural form.  In fact, for the Langevin diffusion on $\mani$, it follows that (see, for example, \cite[Section 2.1]{CGWW}), that for $\xi:\mani \rightarrow \rr$, $\xi \in D(\cdc):=W^{1,2}_{loc}$,  
\[
\cdc(\xi) = \int \norm{\nabla_g \xi}^2_g(x) d\mu(x), \quad .   
\]
Hence, since both the primal and the dual MLD are simply different coordinate representations of the same diffusions, the Dirichlet energies $\cdc_P$ and $\cdc_D$ (and their corresponding domains) are all given by $\cdc$ (and its domain) via the corresponding coordinate charts. 
Thus, although the primal and the dual MLD \eqref{eq:sdemlddual} may superficially have different generators and associated Dirichlet energies, these are the same modulo a coordinate transformation. The following remark, underlining this duality, is important later.

\begin{remark}\label{rmk:dualityFI}
    The primal \eqref{eq:sdemld} and the dual \eqref{eq:sdemlddual} MLD are two coordinate representations of the same Langevin diffusion on the Hessian manifold. Thus they share the same generator, semigroup and the carr\'e-du-champ operator, up to a coordinate transformation, and the same energy function. Thus, the primal MLD satisfies a functional inequality if and only if the dual MLD satisfies the same functional inequality with the same constant.  
\end{remark}

Let us start by recalling the common functional inequalities and their consequences. These definitions can be found in \cite{BakryGentilLedoux2014} Definition 4.2.1 and Definition 5.1.1.

\begin{definition}\label{defn:functioanlineq}
     $\cdc$ is said to satisfy a Poincar\'e inequality (PI) with a constant $c_P$ if, for all functions $\gamma \in D(\cdc)$,  $\mathrm{Var}_{\mu}(\gamma) \le c_P \cdc(\xi)$. Here, $\mathrm{Var}_{\mu}(\gamma):=\int \left( \gamma - \int \gamma d\mu \right)^2 d\mu $ is the variance of the function $\gamma$ under the measure $\mu$. 

    Similarly, $\cdc$ is said to satisfy the logarithmic Sobolev inequality (LSI) with a constant $c_L$ if, for all functions $\gamma \in D(\cdc)$, 
    \[
        \mathrm{Ent}(\gamma^2):= \int \gamma^2 \log \gamma^2 d\mu - \int \gamma^2 d\mu \log\left(\int \gamma^2 d\mu \right)\le c_L \cdc(\gamma). 
    \]
    \end{definition}

These functional inequalities lead to the following corollaries. See \cite[Theorems 4.2.5 and 5.2.1]{BakryGentilLedoux2014}.

\begin{enumerate}
\item[(i)] If $\cdc$ satisfies a Poincar\'e inequality with constant $c_P$, for any function $f \in \mathbf{L}^2(\mu)$, 
\[
\var_\mu(P_t f) \le e^{-2t/c_P} \var_\mu(f). 
\]
\item[(ii)] If $\cdc$ satisfies LSI with constant $c_L$, it converges to $\mu$ exponentially fast in KL in the sense $\KL(\rho_t \mid \mu) \le e^{-2t/c_L} \KL(\rho_0 \mid \mu)$. 
\end{enumerate}

Recall (see, for example, \cite[eq. (1.5)]{CGWW}) that a function $\xi$, that takes values in $[1, \infty)$ is called a Lyapunov function if, for some positive function $\varphi$ such that $\inf \varphi >0$, and positive constants $r_0, b_0$,  
\begin{equation}\label{eq:lyapunov}
\genmld \xi(x) \le - \varphi(x) \xi(x) + b_01\{x:\norm{x} \le r_0\}, \quad \text{for all}\; x \in \rr^d.  
\end{equation}

It is well-known in the literature on Markov processes that the existence of a suitable Lyapunov function implies functional inequalities such as Poincar\'{e} inequalities, log-Sobolev inequalities (LSI) and, consequently, exponential convergence rates of the corresponding Markov process. Our next results are based on the papers \cite{CG10, CG16}.

 The following lemma is a consequece of the assumed diffeomorphism of the map $x\leftrightarrow x^*=\nabla u(x)$.

\begin{lemma}\label{lem:diffeo}
    For every $r>0$, there exists an $s:=s(r)>0$ (depending on $u$) such that $\{x:\; \norm{x^*} \le r\} \subseteq \{x: \norm{x} \le s\}$. Conversely, for every $s>0$, there exists an $r>0$ such that $\{x: \norm{x} \le s\} \subseteq \{x:\; \norm{x^*} \le r\} $. 
\end{lemma}


\begin{theorem}\label{prop:PI}
    Suppose that a weak solution of the MLD exists and
    \begin{enumerate}
    	\item[(a)] $V, F$ are in $C^1$. $u, u^*$ are strictly convex and in $C^2$. 
        \item[(b)] $g=\nabla^2 u$ is locally bounded and locally uniformly elliptic on $R^d$. That is, for any closed ball $B(0,r)$ of radius $r>0$, there exist positive constants $c_r, C_r$ such that $c_r I \leq g(x) \le C_r I$, for all $x\in B(0,r)$.  
        \item[(c)] For some $\alpha \in (0,1]$, 
    \begin{equation}\label{eq:lyapunov_condition}
    \lim_{x\rightarrow \infty}\left[  x \cdot \frac{\partial V}{\partial x} - \alpha x^T \nabla^2u(x) x\right]=\infty.
    \end{equation}
    \end{enumerate}
    Then the MLD Dirichlet energy $\cdc$ satisfies a Poincar\'e inequality for some constant $\lambda >0$. 
\end{theorem}

The constant may be explicitly computed (as will be clear from the proof) but complicated and we cannot properly track its dependence on the dimension. Note that we do not require neither the source measure $\mu=e^{-V}$ nor the target measure $\nu=e^{-F}$ to be log-concave. Condition \eqref{eq:lyapunov_condition} depends only on the tails of $V$ and the mirror map $u$ and is similar to the condition \eqref{eq:wkexistence2} for the weak existence of MLD.

In preparation for the proof, for any $\alpha>0$ define the functions 
\begin{equation}\label{eq:lyapunovchoice}
\xi_\alpha(x)=e^{\alpha u(x)}, \quad \text{and}\quad \varphi_\alpha(y):=  y \cdot \frac{\partial F}{\partial y} - \alpha y^T \nabla^2u^*(y) y. 
\end{equation}

Note that $\xi_\alpha$ is a function on the primal space, while $\varphi_\alpha$ is a function on the dual space. The importance of this pair of functions come from the following observation.
Since 
\begin{equation}\label{eq:lyaderiv}
\nabla \xi_\alpha (x) = \alpha \xi_\alpha \nabla u(x), \quad \nabla^2 \xi_\alpha(x) = \alpha^2 \xi_\alpha \nabla u \otimes \nabla u + \alpha \xi_\alpha \nabla^2 u,
\end{equation}
it follows that
\[
\begin{split}
    \xi_\alpha^{-1}\genmld \xi_\alpha(x) &=- \alpha \frac{\partial u}{\partial x} \cdot \frac{\partial F}{\partial x^*}(x^*) + \alpha^2  \left(\nabla u(x)\right)^T \left(\nabla^2 u(x) \right)^{-1} \nabla u(x) + \alpha  d\\
    &=- \alpha x^* \cdot \frac{\partial F}{\partial x^*}(x^*) + \alpha^2  (x^*)^T \nabla^2 u^*(x^*) x^* + \alpha d\\
    &= - \alpha \left( \varphi_\alpha(x^*) - d\right). 
\end{split}
\]
Note that $u\ge 0$ may be assumed without loss of generality since $u$ is assumed to be strictly convex and therefore attains its infimum. Hence, one can add a constant to $u$ to make it nonnegative. Since all calculations depend on $u$ via its derivatives, this addition make no difference to the argument. However, this turns $\xi_\alpha \ge 1$.

Suppose that $\lim_{y\rightarrow \infty} \varphi_\alpha(y)=\infty$. Fix any $\lambda >0$. Under the above assumptions, there exists an $s>0$ such that,  $\{y:\; \norm{y} > s\} \subseteq \{y:\; \varphi_\alpha(y) > d +\lambda\}$. By Lemma \ref{lem:diffeo}, there is an $r >0$ such that $\{x:\; \norm{x} > r \} \subseteq \{x:\; \norm{x^*} > s\} \subseteq \{x:\; \varphi_\alpha(x^*) > d +\lambda\}$. Hence, for any $x \notin B(0, r)$, $\genmld \xi_\alpha(x) \le -\alpha \lambda  \xi_\alpha(x)$. For $x \in B(0, r)$, 
\begin{equation}
\begin{split}
\genmld \xi_\alpha(x) &= -\alpha (\varphi_\alpha(x^*) - d)\xi_\alpha(x)\\
&= -\alpha\lambda \xi_\alpha(x) + \left( \alpha\lambda \xi_\alpha(x) -\alpha (\varphi_\alpha(x^*) - d)\xi_\alpha(x)\right).
\end{split}
\end{equation}
Let $b_\lambda :=  \sup_{x\in B(0,r)}\left( \alpha\lambda \xi_\alpha(x) -\alpha (\varphi_\alpha(x^*) - d)\xi_\alpha(x)\right)\vee 0$, 
then 
\[
\genmld \xi_\alpha(x) \le  -\alpha \lambda  \xi_\alpha(x) + b_\lambda 1\{ B(0,r)\}.
\]
Thus $\xi_\alpha$ is a Lyapunov function for the primal MLD for every $\lambda>0$ with a constant $b_\lambda$.


\begin{proof}[Proof of Theorem \ref{prop:PI}]  Write the corresponding functions to \eqref{eq:lyapunovchoice} for the dual MLD \eqref{eq:sdemlddual}. Replace $\xi$ and $\varphi_\alpha$ by the corresponding functions 
\begin{equation}\label{eq:duallyapunov}
\xi_\alpha^*(y)=e^{\alpha u^*(y)}, \quad \varphi_\alpha^*(x)= x \cdot \nabla V - \alpha x^T \nabla^2 u(x) x. 
\end{equation}
As remarked above, without loss of generality, we may assume $u^*\ge 0$ so that $\xi_\alpha^* \ge 1$. By our assumptions for Theorem \ref{prop:PI}, $\varphi_\alpha^*$ is continuous and $\lim_{x\rightarrow \infty} \varphi_\alpha^*=\infty$. Following the argument above for $\xi_\alpha$, we see that $\xi^*$ is a Lyapunov function now for the dual MLD process \eqref{eq:sdemlddual}.

We now use \cite[Theorem 4.6.2]{BakryGentilLedoux2014}. We already have a Lyapunov function $\xi^*_\alpha$. It remains to show that the local Poincar\'e inequality holds. The stationary measure for the dual MLD is $\nu=e^{-F}$. Since $F$ is continuous and positive everywhere, over any closed ball $B(0, r)$, $e^{-F}$ is bounded above and below by positive constants. Since the local Poincar\'e inequality is well-known to hold for the uniform measure on $B(0,r)$, thus (see \cite[Proposition 4.2.7]{BakryGentilLedoux2014}), there exists a constant $c'_\lambda$ such that the local Poincar\'e inequality holds for $\nu$, restricted to $B(0,r)$, for the energy function $\int_{B(0,r)} \norm{\nabla \xi^*_\alpha(y) }^2 d\nu$. 

By our assumption $\nabla^2 u^*$ is locally bounded. Thus, for any $r>0$, there exists a constant $c'_r>0$ such that
\[
\norm{\nabla_g \xi^*_\alpha}^2_g(y)= (\nabla \xi^*_\alpha)^T (\nabla^2 u^*)^{-1}(y) \nabla \xi^*_\alpha \ge c'_r  \norm{\nabla \xi^*_\alpha(y) }^2, \quad \forall\; y \in B(0,r). 
\]
This proves that $\nu$ satisfies a local Poincar\'e inequality for the energy function $\cdc_D$ with a dimension-dependent constant. Thus, by \cite[Theorem 4.6.2]{BakryGentilLedoux2014}, the dual MLD satisfies a Poincar\'e inequality over the entire $\rr^d$. By duality, Remark \ref{rmk:dualityFI}, the primal MLD therefore also satisfies a Poincar\'e inequality with the same constant. This completes the proof. 

The constant may be evaluated for each $\lambda >0$ by using \cite[Theorem 4.6.2]{BakryGentilLedoux2014} and then, one may take an infimum over $\lambda$ to get the optimal choice. It is clear that the constant depends on the dimension. 
\end{proof}

Similar conditions give logarithmic-Sobolev inequality. The proof is very similar to that of \cite[Proposition 3.5]{CG16} and is relegated to the Appendix.

\begin{theorem}\label{prop:LSI}
    Assume that all the conditions of Theorem \ref{prop:PI} hold. Additionally, assume that 
    \begin{enumerate}
        \item[(a)] $\nabla^2 u$ is uniformly bounded above by $\alpha_0 I$, for some $\alpha_0 >0$. 
        \item[(b)] For some $\delta, \alpha_1 >0$, 
\begin{equation}\label{eq:lyapunov_conditionLSI}
    \varphi_\alpha^*(x):=\left[  x \cdot \frac{\partial V}{\partial x} - \alpha x^T \nabla^2u(x) x \right]\ge \delta V(x) - \alpha_1.
    \end{equation}
    \item[(c)] $\lim_{x\rightarrow \infty} V(x)=\infty$ and for some $a>0$, $\abs{\nabla V(x)} > a$, for all $x$ large enough.
    \end{enumerate}
    Then the MLD Dirichlet energy $\cdc$ satisfies a logarithmic Sobolev inequality. 
\end{theorem}

Condition \eqref{eq:lyapunov_conditionLSI} says that, outside of a compact set, $\left[  x \cdot \frac{\partial V}{\partial x} - \alpha x^T \nabla^2u(x) x \right]\ge \delta V(x)$ which allows us to compare their level sets.


Let us work out some examples where \eqref{eq:lyapunov_condition} and \eqref{eq:lyapunov_conditionLSI} are satisfied.

\begin{example}
     For our first example let $V(x)=\norm{x}^\beta$, $\beta >1$, and $u(x)=\norm{x}^{\beta'}$, $\beta' \in (1,\min(2,\beta))$. Thus $\abs{\nabla V(x)}=\beta \norm{x}^{\beta-1}$ and $\nabla^2 u(x)$ are uniformly bounded, below and above, respectively, away from the origin. Both of these require a mollification at the origin for their Hessian to be well-defined that we ignore since we are interested in the asymptotic behavior as $x\rightarrow \infty$. Since $x \cdot \nabla V(x)= \beta \norm{x}^\beta$ and $x^T \nabla^2 u(x) x= \beta'(\beta' -1) \norm{x}^{\beta'}$,
    if we take $\alpha=1$, since $1\le \beta' \le \min(\beta,2)$, 
    \[
    \lim_{x\rightarrow \infty}\left[ x \cdot \nabla V(x) - \alpha x^T \nabla^2 u(x) x \right]= \lim_{x\rightarrow \infty}\left[ \beta \norm{x}^\beta - \beta'(\beta'-1) \norm{x}^{\beta'}\right]=\infty.
    \]
    Thus \eqref{eq:lyapunov_condition} for PI holds. Since $\beta > \beta'(\beta'-1)$, for any small enough $\delta$, condition \eqref{eq:lyapunov_conditionLSI} for LSI is also satisfied.   
\end{example}

\begin{example}
    For our next example, take $V(x)=\frac{1}{2}\norm{x}^2$ and $u=\log \sum_{i=1}^d e^{x_i}$. Then,
    \[
    \frac{\partial u}{\partial x_i}:=w_i=\frac{e^{x_i}}{\sum_{j=1}^d e^{x_j}}, \quad x^T\nabla^2 u(x) x = \var_w(x),  
    \]
    where $\var_w(x)$ is the variance of the vector $x$ with weights $w$, i.e., 
    \[
    \var_w(x)= \sum_{i=1}^d w_i x_i^2 - \left( \sum_{i=1}^d w_i x_i\right)^2.
    \]
    Note that $w$ is always a probability vector.
    
    Hence, $x\cdot \nabla V(x) - \alpha x^T \nabla^2 u(x) x =  \norm{x}^2 - \alpha \var_w(x)$. 
    Since $\norm{x}^2 \ge \sum_{i=1}^d w_i x_i^2\ge \var_w(x)$, it follows that, for any $\alpha \in (0,1)$, \eqref{eq:lyapunov_condition} is true. \eqref{eq:lyapunov_conditionLSI} is also true for $\delta=1$.  Although $e^{-V}$ is trivial to simulate this example shows that a nontrivial MLD also converges exponentially fast. 
\end{example}

\section{A Gibbs sampling Markov chain for the MLD}\label{sec:gibbsmc}

In this section we describe a family of Markov chains, inspired by the Sinkhorn algorithm used in entropy-regularized optimal transport, that approximates the primal MLD with stationary density $\mu=e^{-V}$ \eqref{eq:sdemld}. The family is parametrized by a scalar parameter $(\epsilon >0)$. For every $\epsilon>0$, the corresponding Markov chain has a unique stationary distribution $\mu$, and the process laws of the Markov chain converges to that of the (primal) MLD as $\epsilon \downarrow 0$. As always, we are going to assume that there is a weak solution of the MLD (see \eqref{eq:wkexistence} and \eqref{eq:wkexistence2}). 


Recall that for a point in $\mani$ with primal coordinate $x$, $x^*=\nabla u(x)$ denotes its dual coordinate. Conversely, for a point with dual coordinate $y$, $y_*=\nabla u^*(y)$ refers to its primal coordinate. Sometime we will use other letters; it will be clear from the context if those refer to primal or dual coordinates. Recall the notation:
    $
        \frac{\partial x^*}{\partial x}=\nabla^2 u(x), \quad \frac{\partial x}{\partial x^*}=\nabla^2 u^*(x^*)=\left(\nabla^2 u(x)\right)^{-1}. 
    $

Define a family of Gaussian conditional densities, 
\[
\begin{split}
q_\epsilon(y \mid x) &= \text{density of}\; N\left( x^*, \eps \frac{\partial x^*}{\partial x}\right)\\
&=\frac{1}{(2\pi \eps)^{d/2}\left(\det g(x)\right)^{1/2}}\exp\left[-\frac{1}{2\eps}(y-x^*)^T \frac{\partial x}{\partial x^*}(y-x^*) \right].
\end{split}
\]
Consider the joint density 
\begin{equation}\label{eq:whatispi}
\pi_\eps(x,y)= e^{-V(x)}q_\eps(y \mid x). 
\end{equation}
Then, under $\pi_\eps$, the conditional density of $X$, given $Y=y$, is given by 
\[
\hat{q}_\epsilon(w \mid y) = \frac{1}{\int e^{-V(x)} q_\epsilon(y \mid x)dx }e^{-V(w)} q_\epsilon(y \mid w).
\]

\begin{definition}\label{defn:mcdef}
    Define the Markov transition density 
    \[
    r_\epsilon(z\mid x)= \int q_\epsilon(y\mid x) \hat{q}_\epsilon(z \mid y)dy.
    \]
    That is, $r_\epsilon(\cdot \mid x)$ is the conditional density of $Z$, given $X=x$, where $(X, Y, Z)$ is a two-step Markov chain where the density of $Y$, given $X=x$, is $q_\epsilon(\cdot \mid x)$, and that of $Z$, given $Y=y, X=x$, is $\hat{q}_\eps(\cdot \mid y)$. 
\end{definition}

As explained in the Introduction, our Markov chain is both a Gibbs sampler and related to the Sinkhorn algorithm. Hence we can immediately guess that the KL divergence from $\mu$ must be monotonically decreasing. That is, if $\rho_k^\eps$ denotes the law of $X_k^\eps$, starting with some initial $\rho_0^\eps$, the map $k \mapsto \KL(\rho_k^\eps\mid \mu)$ must be a non-increasing function. This argument is well-known, but we repeat this anyway.

It suffices to show that one step transition of the Markov chain cannot increase Kullback-Leibler divergence. Suppose $X \sim \rho_0$, $Y\sim q_\eps(\cdot \mid x)$, given $X=x$ and $Z\sim \hat{q}_\eps(\cdot \mid y)$, given $X=x, Y=y$. If $Z \sim \rho_1$, we claim $\KL(\rho_1 \mid \mu) \le \KL(\rho_0 \mid \mu)$. Consider the joint density $\gamma_0(x,y)=\rho_0(x) q_\eps(y\mid x)$. Obviously, $\KL(\gamma_0 \mid \pi_\eps) = \KL(\rho_0 \mid \mu)$. Let $\nu_0$ and $\nu$ denote the marginal density of the $Y$ coordinate under $\gamma_0$ and $\pi_\eps$, respectively. Then, one also gets $\KL(\rho_0 \mid \mu)=\KL(\gamma_0 \mid \pi_\eps) = \KL(\nu_0 \mid \nu) + \E \left[\KL(\hat{q}_\eps^0(\cdot \mid Y), \hat{q}_\eps(\cdot \mid Y))\right]$, where $\hat{q}_\eps^0(\cdot \mid Y=y)$ is the conditional density of $X$, given $Y=y$, under $\gamma_0$. Thus $\KL(\nu_0 \mid \nu) \le \KL(\rho_0 \mid \mu)$. Repeating this argument again, going from $Y$ to $Z$, shows that $\KL(\rho_1 \mid \mu) \le \KL(\nu_0 \mid \nu) \le \KL(\rho_0 \mid \mu)$. This shows monotonicity of KL without providing any rate of decay. 

Other notions of convergence (without explicit rates) are immediate too. 

\begin{theorem}\label{thm:stationarymc}
    Let $\left( X^\epsilon_k,\; k=0,1,2,\ldots \right)$ be the Markov chain with initial distribution $\rho_0$ and transition density $r_\epsilon$. Then, for any $\rho_0$, the Markov chain $X^\epsilon$ converges in total variation to the unique stationary distribution $\mu$.
\end{theorem}

\begin{proof}
    It is obvious that $\mu$ is a stationary distribution for the Markov chain. 
    

    Since the transition probabilities are absolutely continuous and $\mu$ is supported everywhere, the Markov chain is $\mu$ irreducible according to \cite[Section 4.2.1]{meyntweediebook}. Since $\mu$ is a probability measure, the Markov chain is a positive Harris recurrent chain by \cite[Proposition 10.1.1]{meyntweediebook} and the invariant probability measure is unique by \cite[Theorem 10.4.4]{meyntweediebook}. Ergodicity now follows from \cite[Theorem 13.0.1]{meyntweediebook}. This proves the theorem. 
\end{proof}

Our nontrivial contribution, beside the following diffusion approximation result, is the derivation of an explicit convergence rate for this Markov chain that is consistent with its diffusion limit. 


\begin{assumption}\label{asmp:MCconv}
Suppose that the following conditions are satisfied. 
\begin{enumerate}[(i)]
        \item $u, u^*$ are in $C^6$ with all their derivatives up to the sixth order uniformly bounded.
        \item There are positive constants $c_0$ and $C_0$ such that, for all $x\in \rr^d$, $C_0 I \succcurlyeq \nabla^2 u(x) \succcurlyeq c_0 I$. In particular, a similar upper and lower bound holds for $\nabla^2 u^*$ as well. 
        \item $V$ is twice differentiable all its second derivatives are uniformly bounded.
    \end{enumerate}
\end{assumption}

For the following result we require an additional assumption on the primal MLD.    Assume that the martingale problem for the SDE \eqref{eq:sdemld} is well-posed. That is, roughly, there is a weak solution that is \textit{unique} for every starting position. 
Note that we have already assumed weak existence. 

\begin{theorem}\label{thm:mcconv}
    For each $\epsilon >0$, let $\left( X^\epsilon_k,\; k=0,1,2,\ldots \right)$ be a Markov chain with transition density $r_\epsilon$ and initial value $X^\eps_0=x^\eps_0$. Assume that $\lim_{\eps \rightarrow 0+} x^\eps_k=x_0$. Then, under Assumption \ref{asmp:MCconv}, the continuous time interpolated process $X^\epsilon_t:= X^\epsilon_k, \quad k=\lfloor t/\epsilon \rfloor$, $t\ge 0$, converges weakly in the Skorokhod topology to the law of the MLD \eqref{eq:sdemld} with initial value $x_0$. 
\end{theorem}

Recall that the chi-square divergence between two probability measures $p$ and $q$ is given by 
\[
\chi^2\left( p \mid q \right)= \int \left( \frac{dp}{dq} - 1\right)^2 dq, 
\]
if $p$ is absolutely continuous with respect to $q$, and infinity otherwise.

\begin{theorem}\label{thm:mclya}
 Let $e^{-F}= (\nabla u)_{\#e^{-V}}$, as before. Suppose the following assumptions hold.
 \begin{enumerate}[(i)]
     \item $e^{-F}$ satisfies a Poincar\'e inequality. 
     \item $F$ is $L$-smooth, i.e., $\norm{\nabla F(y) - \nabla F(x)}\le L \norm{y-x}$, for all $x,y$.
     \item The convex function $u$ satisfies Assumption \ref{asmp:MCconv} (i) and (ii). 
     \item The initial variable $X^\eps_0$ has a density $p^\eps_0$ such that $\chi^2(p^\eps_0 \mid \mu) < \infty$. 
 \end{enumerate}
 Then, there exists a constant $c_0>0$ such that for all $\eps \in (0,1)$, if $p^\eps_k$ is the density of $X^\eps_k$ then 
 \[
 \chi^2\left( p^\eps_k  \mid \mu\right) \le (1 - c_0 \eps)^k \chi^2(p^\eps_0 \mid \mu), \quad \forall\; k \in \mathbb{N}. 
 \]
\end{theorem}
Note that since $\KL$ between two measures is smaller that the chi-square divergence, Theorem \ref{thm:mclya} also implies that $\KL\left( p^\eps_k  \mid \mu\right) \le (1 - c_0 \eps)^k \chi^2(p^\eps_0 \mid \mu)$, i.e., an exponential decay in $\KL$ to equilibrium.  

\begin{remark}
    By the Jacobian formula, $F(x^*) = \log \det \nabla^2 u(x)+ V(x)$. Thus condition (iii) in Assumption \ref{asmp:MCconv} follows if $F$ has second derivatives bounded, which is a slight strengthening of condition (ii) of Theorem \ref{thm:mclya} and condition (i) of Assumption \ref{asmp:MCconv}.
\end{remark}

\begin{remark}
Since the diffusion approximation in Theorem \ref{thm:mcconv} requires a time scaling by $\eps$, the convergence rate in Theorem \ref{thm:mclya} is consistent with the diffusion limit. That is, suppose the limiting MLD diffusion converges exponentially fast in $\KL$ or $\chi^2$ divergence. Then, given any $\delta >0$, it is $\delta$ close to equilibrium in $O(\log(1/\delta))$ time in continuum. This corresponds to $O(\eps^{-1}\log(1/\delta))$ many steps of the approximating Markov chain. Theorem \ref{thm:mclya} gives a contraction of $(1- c_0 \eps)^{O(\eps^{-1}\log(1/\delta))}=e^{-O(\log \delta)}$. Thus the Markov chains and the limiting diffusion have comparable convergence rates when the time scale is measured in step size $\eps$.    \end{remark}

\begin{example}
   When $u(x)=\frac{1}{2}\norm{x}^2=u^*(x)$, $\nabla^2 u=I$. Thus $q_\eps(\cdot \mid x)=N(x, \epsilon I)$, the transition density of $d$-dimensional Brownian motion. In this case the Hessian geometry reduces to the Euclidean geometry and the MLD reduces to the usual Langevin diffusion. Theorem \ref{thm:mcconv} proposes a Markov chain discretization for the Langevin diffusion with the stationary distribution and guaranteed convergence rate (under suitable assumptions). As far as we know, this is a novel contribution to the literature.  
\end{example}

\begin{proof}[Proof of Theorem \ref{thm:mcconv}]
The proof of Theorem \ref{thm:mcconv} will follow by an application of \cite[Theorem 8.7.1]{durrett1996stochastic}. In order to show that the assumptions of the cited Theorem are satisfied, we will verify conditions (i), (ii) and (iii) in \cite[Theorem 8.7.1]{durrett1996stochastic} via a series of lemmas whose proofs are given below this one.

Recall that the pushforard of the density $\mu=e^{-V}$ via the map $x \mapsto x^*$ is given by the density $\nu=e^{-F}$.

    \noindent\textbf{Step 1.} (Verifying condition (i) in \cite[Theorem 8.7.1]{durrett1996stochastic}) Let $X^\eps_1$ be sampled from the conditional density $r_\epsilon(\cdot \mid x)$. Then, we show that, for every $R >0$ 
    \[
       \lim_{\eps \rightarrow 0+}\sup_{\norm{x} \le R} \eps^{-1}\left(\E(X_1^\eps) - x\right) = -  \frac{\partial F}{\partial x^*}(x^*).
    \]
    
    After rescaling time this is consistent with the drift of the diffusion \eqref{eq:sdemld} which, informally, says that 
    \[
     \E(X_\epsilon \mid X_0=x) = x - \epsilon \frac{\partial F}{\partial x^*}(x^*) + o(\epsilon).
    \]
    
   \noindent\textbf{Step 2.}  (Verifying condition (ii) in \cite[Theorem 8.7.1]{durrett1996stochastic}) Let $\Cov$ represent the covariance matrix of a random vector. We show that, for every $R>0$, 
    \[
        \lim_{\eps \rightarrow 0+}\sup_{\norm{x}\le R}\norm{\eps^{-1}\Cov(X_1^\eps) -  2 \frac{\partial x}{\partial x^*}} = 0. 
    \]
    This is also consistent with the diffusion matrix of \eqref{eq:sdemld} which may be interpreted as
    \[
        \Cov(X_\epsilon \mid X_0=x) = 2\epsilon \frac{\partial x}{\partial x^*} + o(\epsilon).
    \]
    
    \noindent\textbf{Step 3.}  (Verifying condition (iii) in \cite[Theorem 8.7.1]{durrett1996stochastic}). We show that, for all $\delta>0$,
    \[
     \lim_{\eps \rightarrow 0+}\sup_{\norm{x} \le R}\eps^{-1}P\left( \norm{X_1^\eps - x} > \delta \right) = 0.
    \]
 
Theorem \ref{thm:mcconv} then follows from \cite[Theorem 8.7.1]{durrett1996stochastic}.
\end{proof}   


We now prove Steps $1-3$. Our main tool is the following general Laplace approximation result. Fix some $y\in \rr^d$. Let $\Psi$ be the function 
\begin{equation}\label{eq:whatispsi}
\Psi(w)=(w^* - y)^T \frac{\partial w}{\partial w^*} (w^*-y).
\end{equation}
Note that, $\Psi$ depends on $y$, but we drop it from the notation for clarity. In the following argument, $y$ will be treated as a constant. 

Let $w_0=y_*$, whereby $w_0^*=y$. Thus $\Psi$ is a nonnegative function with a unique minimum at $w_0$. We have assumed that both $u$ and $u^*$ are in $C^6$. Hence $\nabla^2 u^*$ is assumed to be $C^4$. Hence, $\Psi$ admits first three derivatives as calculated below.

Obviously $\Psi(w_0)=0$ and $\nabla \Psi(w_0)=0$ by optimality. We claim that $\nabla^2 \Psi(w_0):=2H$ is positive definite. To see this, compute 
\[
\begin{split}
\nabla_w \Psi(w)&= 2 \frac{\partial w^*}{\partial w} \frac{\partial w}{\partial w^*} (w^*-y) + (w^* - y)^T \nabla_w \frac{\partial w}{\partial w^*} (w^*-y)\\
&= 2(w^* - y) + (w^* - y)^T \nabla_w g^{-1}(w) (w^*-y).
\end{split}
\]
Note that, since $w_0^*=y$, $\nabla \Psi(w_0)=0$, as expected. For the Hessian, fix $j \in [d]$. Let $\partial_j=\frac{\partial}{\partial w_j}$. Then, if $g_{*j}$ refers to the $j$th column of the matrix $\frac{\partial w}{\partial w^*}$, then
\[
\begin{split}
\nabla_w &\partial_j\Psi(w)= 2 g_{*j}(w) + \nabla_w \left[ (w^* - y)^T \partial_j g^{-1}(w) (w^*-y)\right]\\
&= 2 g_{*j}(w) + 2g(w) \partial_j g^{-1}(w) (w^*-y) +  (w^* - y)^T \nabla_w \partial_j g^{-1}(w) (w^*-y).
\end{split}
\]
In particular, $2H=\nabla^2_w \Psi(w_0)=2g(y_*)$. Thus $H=\frac{\partial y}{\partial y_*}$, which by assumption is positive definite. 

Our nontrivial observation is that all third derivatives of $\Psi$ vanish at $w_0=y_*$. To see this, let us compute the partial derivative of the $(i,j)$th element of the Hessian at $w_0=y_*$. Note that any term that contains a factor of $(w^*-y)$ will vanish when evaluated at $w=w_0$. For example, the derivative of the final term $(w^*-y)^T\nabla^2_w g^{-1}(w) (w^*-y)$ vanishes at $w_0$. Thus, by ignoring that term,
\[
\partial^3_{ijk}\Psi(w)\mid_{w_0}=2 u_{ijk}(y_*) + 2 g_{il}g_{mk} \partial_j g^{ml}(y_*). 
\]
Now use the matrix identity 
\begin{equation}\label{eq:derivinverse}
\partial_j g^{ml}= -g^{mp}g^{lq} \partial_j g_{pq}.
\end{equation}
Thus 
\[
\begin{split}
\partial^3_{ijk}\Psi(w)\mid_{w_0}&=2 u_{ijk}(y_*) - 2 g_{il}g_{mk} g^{mp}g^{lq} u_{jpq}(y_*)\\ 
&=2 u_{ijk}(y_*) - 2 g_{il} g^{lq} g_{mk} g^{mp} u_{jpq}(y_*)\\
&=2 u_{ijk}(y_*) - 2\delta_{iq}\delta_{kp} u_{jpq}(y_*)= 0.
\end{split}
\]
Thus all the third derivatives of $\Psi$ vanishes at $w_0$, as claimed.

For a $C^1$ function $\phi$, consider a probability density 
\[
   p_\eps(w) = \frac{1}{C_\epsilon}\exp\left( -\frac{1}{ 2\eps} \Psi(w) - \phi(w) \right),
   \]
where $C_\epsilon$ is the normalizing constant (also depends on $y$ but suppressed from the notation). 

In the remainder of this section we freely use the notations $O(\eps^k)$ and $o(\eps^k)$, for some $k \in \mathbb{N}$, without specifying the constants. This is allowed by our assumption that all derivatives of relevant quantities are bounded above by positive constants. Thus, the only dependence we track is the magnitude of their dependence on $\eps$.  

For example, we are interested in computing the first two moments under $p_\eps$ up to a an error of $o(\eps)$. This is achieved by the following lemmas. 

\begin{lemma}\label{lem:gennormconstant}
 Assume that all second derivatives of $\phi$ are uniformly bounded and that $\inf_x \phi(x) > -\infty$. Then, with $H:= \frac{\partial y}{\partial y_*}$, 
$$
 \lim_{\eps \rightarrow 0} \eps^{-d/2} C_\eps =  e^{-\phi(w_0)}\left[\det (H) \right]^{-1/2}.
$$

\end{lemma}

\begin{proof}
    Obviously,
    \[
    C_\epsilon = \int_{\rr^d} \exp\left( -\frac{1}{2\eps} \Psi(w) - \phi(w) \right) dw= \int_{\rr^d} \exp\left( -\frac{1}{2\eps} \Psi_\eps(w) \right) dw,
    \]
    where $\Psi_\epsilon(w)=\Psi(w) + 2\epsilon \phi(w)$.
    We claim that the integral has an exponentially small mass beyond a compact set. 

     Since, by assumption, $\inf\phi > - M$ and $\inf \nabla^2 u^* \ge m I$, for some $m, M > 0$,
    \[
        \Psi_\eps(w) > m \norm{w^* - y}^2 - 2\eps M. 
    \]
    The choices of $m$ and $M$ are independent of $y$. Hence, for any $c >0$, there exists a radius $\alpha \sqrt{\eps} > 0$ such that, for all $w \notin B_{\alpha \sqrt{\epsilon}}(y_*)$ and all $\eps$ small enough, 
    \begin{equation}\label{eq:expsmallballcomp}
        \eps^{-d/2}\int_{w \notin B_{\alpha \sqrt{\epsilon}}(y_*)} \exp\left( -\frac{1}{2\eps} \Psi_\eps(w)\right)dw < e^{-c}.  
    \end{equation}

    Now, inside the ball $B_{\alpha \sqrt{\eps}}(y_*)$ do a Taylor expansion of $\Psi_\eps(w)$ in terms of the variable $z=(w - y_*)/\sqrt{\eps}$. 

    
    First consider the function $\Psi$. Recall that $w_0=y_*$ is the unique minimizer of $\Psi$. We have already shown $\Psi(w_0)=0, \nabla \Psi(w_0)=0$, $\nabla^2 \Psi(w_0)=2H$, $\nabla^3 \Psi(w_0)=0$ and $\Psi$ has bounded fourth derivatives. Expanding $\Psi$ by a third order Taylor expansion around $w_0=y_*$ gives 
    \begin{equation}\label{eq:expandPsi}
    \begin{split}
    \Psi(w)&=(w-w_0)^T H (w-w_0)  + O(\norm{w-w_0}^4)\\
    &=\eps z^T H z + \eps^2 O(\norm{z}^4)= \eps z^T H z + O(\eps^2), 
    \end{split}
    \end{equation}
    since $\norm{z}\le \alpha$ for $w\in B_{\alpha \sqrt{\eps}}(w_0)$.
    For the lower order term $2\eps \phi$, it suffices to only consider the first order Taylor approximations and use the assumption that the second derivatives are uniformly bounded.  Then $\phi(w) = \phi(y_*) + \sqrt{\eps} z\cdot \nabla \phi(y_*)+ O(\eps)$.


    This gives us $\Psi_\eps(w)$ is equal to
    \begin{equation}\label{eq:psieps3} 
    \eps z^T H z 
    + 2\eps \phi(y_*)  
    + 2\eps^{3/2} z \cdot \nabla \phi(y_*) + O(\eps^2).
    \end{equation}
    
    Hence, combining everything,
    \[
    \begin{split}
        \int_{B_{\alpha \sqrt{\epsilon}}(y_*)} &\exp\left( -\frac{1}{2\eps} \Psi_\eps(w)\right)dw= e^{-\phi(y_*)} \eps^{d/2} \int_{B(0, \alpha)} \exp\left(- \frac{1}{2}z^T H z  \right) \mathcal{R}_\eps(z)dz,
    \end{split}
    \]
    where
    \begin{equation}\label{eq:whatisreps}
    \begin{split}
    \mathcal{R}_\eps(z)&= \exp\left( -\frac{1}{2\eps}\Psi(w) + \frac{1}{2}z^T Hz -  \left( \phi(w) - \phi(y_*)\right)\right)\\
    &=\exp\left( -\eps^{1/2} z \cdot \nabla_w \phi(y_*) + O(\eps) \right).
    \end{split}
    \end{equation}

    Since we are integrating over a bounded ball, we can approximate the exponential by a first order Taylor expansion,
    \begin{equation}\label{eq:taylorreps}
    \mathcal{R}_\eps(z) = 1 - \eps^{1/2} z \cdot \nabla_w \phi(y_*) +  O(\eps).
    \end{equation}
    Since the linear function is odd, its integral over the ball with respect to the even function $\exp\left(-\frac{1}{2}z^T H z  \right)$ is exactly zero. Hence,
    \[
     \int_{B_{\alpha \sqrt{\epsilon}}(y_*)} \exp\left( -\frac{1}{2\eps} \Psi_\eps(w)\right)dw= (1+ O(\eps)) e^{-\phi(y_*)} \eps^{d/2} \int_{B(0, \alpha)} \exp\left(-\frac{1}{2}z^T H z  \right) dz
    \]
    
    Adding back the integral over the complement of the ball, we get that $e^{2\phi(y_*)}\eps^{-d/2}C_\eps$   
    is bounded below by
    \[
    (1+ O(\eps))\int_{B(0, \alpha)} \exp\left(-\frac{1}{2}z^T H z  \right)dz
    \]
    and bounded above by the above plus $e^{-c}$. Take $\eps \downarrow 0$ and then $c\uparrow \infty$, to obtain 
    \[
    \begin{split}
    \lim_{\eps \rightarrow 0} e^{\phi(y_*)}\eps^{-d/2}\int \exp\left( -\frac{1}{2\eps} \Psi_\eps(w)\right)dw &= (2\pi)^{d/2} \left[\det H\right]^{-1/2}.
    \end{split}
    \]
    That is, $\lim_{\eps \rightarrow 0} \eps^{-d/2} C_\eps = e^{-\phi(w_0)}\left[\det  H\right]^{-1/2}$, since $w_0=y_*$. 
\end{proof}

The next lemma computes the first moment of $p_\epsilon$ as $\epsilon \rightarrow 0$. Recall the notion of the Riemannian gradient $\nabla_g$ from Section \ref{sec:hessian_prelim}. 

\begin{lemma}\label{lem:genfirstmoment}
    Under the conditions of Lemma \ref{lem:gennormconstant}, 
    \[
    \lim_{\epsilon \rightarrow 0+} \frac{1}{\epsilon}\left(\int w p_\epsilon(w)dw - w_0\right)= - \nabla_g \phi(w_0). 
    \]
    In other words, $\int w p_\epsilon(w)dw= w_0 - \epsilon \nabla_g \phi(w_0) + o(\epsilon)$.
\end{lemma}

\begin{proof} We follow the notations and ideas in the proof of Lemma \ref{lem:gennormconstant}. Obviously
\[
\int w p_\eps(w) dw = \frac{1}{C_\eps} \int w \exp\left( -\frac{1}{2\eps} \Psi_\epsilon(w \right) dw,
\]
where, as before, $C_\eps$ is the normalizing constant and $\Psi_\epsilon=\Psi + 2\epsilon \phi$. 

As in the proof of Lemma \ref{lem:gennormconstant}, for any $c>0$, depending on $\epsilon$, there exists a function $\alpha= \alpha(\eps):=\log(1/\eps)$ such that, for all small enough $\eps$,
\[
\eps^{-d/2}\int_{w\notin B_{\alpha \sqrt{\epsilon}}(w_0)}  \norm{w} \exp\left( -\frac{1}{ 2\eps} \Psi_\epsilon(w \right) dw < \epsilon e^{-c}. 
\]
Ignore the integral outside the ball $B_{\alpha \sqrt{\epsilon}}(w_0)$. Inside the ball, let $z=(w-w_0)/\sqrt{\eps}$. Then, 
\[
\begin{split}
\int w p_\eps(w) dw &= w_0 + \frac{\eps^{1/2}}{C_\eps} \int z \exp\left( -\frac{1}{2\eps} \Psi_\epsilon(w) \right) dw\\
&\le w_0 + \frac{\eps^{1/2}}{C_\eps} \int_{B_{\alpha \sqrt{\epsilon}}(w_0)} z \exp\left( -\frac{1}{2 \eps} \Psi_\epsilon(w) \right) dw + \epsilon^{d/2+1} e^{-c}.
\end{split}
\]
By the same logic, 
\[
\int w p_\eps(w) dw  \ge  w_0 + \frac{\eps^{1/2}}{C_\eps} \int_{B_{\alpha \sqrt{\epsilon}}(w_0)} z \exp\left( -\frac{1}{ 2\eps} \Psi_\epsilon(w) \right) dw - \epsilon^{d/2+1} e^{-c}.
\]
Since $c$ can be taken arbitrarily large, we can restrict ourselves in estimating the integral inside the vanishingly small ball $B(w_0,  \sqrt{\epsilon} \log(1/\eps))$. 

By Lemma \ref{lem:gennormconstant}, $C_\eps = e^{-\phi(w_0)}(2\pi\eps)^{d/2} \left(\det  H\right)^{-1/2}\left(1 + o(1) \right)$. Substitute the leading term for $C_\eps$. The error due to the $o(1)$ is negligible. Hence, our objective is to show that
\[
\frac{\eps^{1/2} e^{\phi(w_0)} \left(\det H\right)^{1/2}}{(2\pi \eps)^{d/2} } \int_{B(w_0, \alpha \sqrt{\epsilon})} z \exp\left( -\frac{1}{2\eps} \Psi_\epsilon(w) \right) dw = - \epsilon \nabla_g \phi(w_0) + o(\epsilon). 
\]

From \eqref{eq:whatisreps} one may write
\[
\begin{split}
e^{-\frac{1}{2\eps} \Psi(w)}&= \mathcal{R}_\eps(w) e^{- \frac{1}{2}z^T H z} e^{\phi(w) - \phi(y_*)}.\\
\text{I.e,}\; e^{ \phi(y_*)}e^{-\frac{1}{2\eps} \Psi_\eps(w)}&= \mathcal{R}_\eps(w) e^{- \frac{1}{2}z^T H z}.
\end{split}
\]
Since $w_0=y_*$,
\[
\begin{split}
&\frac{\eps^{1/2} e^{\phi(w_0)} \left(\det H\right)^{1/2}}{(2\pi \eps)^{d/2} } \int_{B_{ \alpha \sqrt{\epsilon}}(w_0)} z \exp\left( -\frac{1}{2\eps} \Psi_\epsilon(w) \right) dw = \\
&  \frac{\eps^{1/2}(\det  H)^{1/2}}{(2\pi)^{d/2}} \int_{B_{\alpha}(0)} z \exp\left( -  \frac{1}{2}z^T Hz\right)\mathcal{R}_\eps(w_0+ \sqrt{\eps}z)dz.
\end{split}
\]

Consider the integral
\[
\frac{(\det  H)^{1/2}}{(2\pi)^{d/2}} \int_{B_{\alpha}(0)} z \exp\left( - \frac{1}{2}z^T Hz\right)\mathcal{R}_\eps(w_0+ \sqrt{\eps}z)dz.
\]
As before, approximating $\mathcal{R}_\eps$ inside this ball by 
\[
\mathcal{R}_\eps(z) = 1 - \sqrt{\eps} z\cdot \nabla \phi(y_*) + O(\eps), 
\]
gives 
\[
\begin{split}
&\frac{(\det  H)^{1/2}}{(2\pi)^{d/2}} \int_{B_{\alpha}(0)} z \exp\left( - \frac{1}{2}z^T Hz\right)\mathcal{R}_\eps(w_0+ \sqrt{\eps}z)dz\\
&= \frac{(\det H)^{1/2}}{(2\pi)^{d/2}} \int_{B_{\alpha}(0)} z \exp\left( - \frac12 z^T Hz\right)dz\\
& - \sqrt{\eps} \frac{(\det H)^{1/2}}{(2\pi)^{d/2}} \int_{B_{\alpha}(0)} z z^T  \nabla \phi(y_*) \exp\left( - \frac12 z^T Hz\right)dz\\
&+ \eps \frac{(\det  H)^{1/2}}{(2\pi)^{d/2}} \int_{B_{\alpha}(0)} O(\norm{z}^3) \exp\left( - \frac12 z^T Hz\right)dz. 
\end{split}
\]
Here the $O(\norm{z}^3)$ is due to teh assumption that $\phi$ has all bounded second derivatives. 

Now, as $\eps \downarrow 0$, $\alpha(\eps) = \log(1/\eps) \uparrow \infty$. Thus the above integrals over the ball of radius $\alpha$ may be approximated by the full Gaussian integral. This gives us an RHS $ - \eps^{1/2} (H)^{-1} \nabla \phi(y_*) + O(\eps)$.

Combining all our previous steps, 
\[
\begin{split}
\int wp_\eps(w)dw &= w_0 - \eps (H)^{-1}\nabla \phi(w_0) + o(\eps)\\
&= w_0 - \eps g^{-1}(y_*) \nabla \phi(y_*) + o(\eps)= w_0 - \eps \nabla_g \phi(y_*) + o(\eps),
\end{split}
\]
proving our claim.
\end{proof}

\begin{lemma}\label{lem:gensecondmoment}
Let $\Cov(p_\eps)$ denote the covariance matrix under $p_\eps$. Then, under the conditions of Lemma \ref{lem:gennormconstant}, 
    \[
    \lim_{\epsilon \rightarrow 0+} \frac{1}{\epsilon}\Cov(p_\eps)= (H)^{-1}. 
    \]
    In other words, $\Cov(p_\eps)= \eps(H)^{-1} + o(\epsilon)$.
\end{lemma}

\begin{proof}
    This follows from a similar but simpler argument as in Lemma \ref{lem:genfirstmoment}. We skip the details.
\end{proof}

\begin{proof}[Proof of Theorem \ref{thm:mcconv} Steps 1 and 2]
    The density
    $\hat{q}_\epsilon(\cdot \mid y)$ is given by
    \[
    \begin{split}
        \hat{q}_\eps(w \mid y) & =  \frac{1}{C_\eps(y)} \exp\left( -\frac{1}{2\eps} \Psi(w) - V(w) -\frac{1}{2} \log \abs{\frac{\partial w^*}{\partial w}} \right) \\
        &= \frac{1}{C_\eps(y)}\exp\left( - \frac{1}{2\eps} \Psi(w) - \phi(w) \right)
    \end{split}
    \]
    where
    \begin{equation}\label{eq:psieps2}
    \phi(w)=  V(w) + \frac{1}{2} \log \det \frac{\partial w^*}{\partial w}= V(w) + G(w).
    \end{equation}
    The assumptions in Lemma \ref{lem:gennormconstant} are satisfied by Assumption \ref{asmp:MCconv}.

    Thus, by Lemmas \ref{lem:gennormconstant}, \ref{lem:genfirstmoment} and \ref{lem:gensecondmoment}, we can now estimate the first two moments under $\hat{q}_\eps$. 
    \[
    \int w \hat{q}_\eps(w\mid y)dw = y_* - \epsilon\nabla_g\left( V + G \right)(y_*) + o(\eps),
    \]
    and $\Cov(\hat{q}_\eps(\cdot \mid y))= \eps\frac{\partial y_*}{\partial y} + o(\eps)$. 

    We will now compute the first two moments under $r_\eps(\cdot x)$ by what is sometimes called the delta method. Recall that if $Y \sim N\left(x^*, \eps \frac{\partial x^*}{\partial x}\right)$ and $Z$, given $Y=y$, has density $\hat{q}_\eps(\cdot \mid y)$, then the unconditional density of $Z$ is $r_\eps(\cdot \mid x)$. We will compute $\E(Z)$ and $\Cov(Z)$. By the tower property,
\[
\begin{split}
    \E(Z)&= \E\left[ \E(Z \mid Y) \right] = \E\left[ Y_* - \epsilon \nabla_g(V + G)(Y_*) \right] + o(\eps).
\end{split}
\]
We now estimate the RHS up to an $o(\eps)$ error by a first order Taylor approximation to $y \mapsto y_*=\nabla u^*(y)$ around $x^*$. Using the assumption that $u^*$ has all bounded fourth derivatives,
\begin{equation}\label{eq:taylorstar}
\begin{split}
    \nabla u^*(y)= x + \nabla^2 u^*(x^*) (y - x^*) + \frac{1}{2}(y-x^*)^T \nabla^3 u^*(x^*) (y-x^*) + O(\norm{y-x^*}^3).
\end{split}
\end{equation}
Note that $\nabla^3 u^*(x^*)$ is a third order tensor and the multiplication with the two vectors on either side outputs a vector.

Applying this to $Y \sim N\left(x^*, \eps \frac{\partial x^*}{\partial x}\right)$, for any $k\in [d]$, $\E((Y_*)_k)=$
\[
\begin{split}
& x_k + \frac{1}{2}\sum_i \sum_j \frac{\partial^2 x_k}{\partial x^*_i \partial x^*_j} \E\left[(Y_i - x^*_i)(Y_j - x^*_j)\right] + o(\eps)  \\
&= x_k + \frac{\eps}{2} \left( \frac{\partial}{\partial x_i^*} g^{kj} \right) g_{ij}
+ o(\eps)
= x_k + \frac{\eps}{2} g_{ij} g^{il} \frac{\partial}{\partial x_l} g^{kj} + o(\eps)\\
&= x_k - \frac{\eps}{2} g_{ij} g^{il} g^{kp}g^{jq} u_{lpq} + o(\eps), \quad \text{by \eqref{eq:derivinverse}}, 
\\
&= x_k - \frac{\eps}{2} \delta_{jl} g^{kp}g^{jq} u_{lpq} + o(\eps)=x_k - \frac{\eps}{2} g^{kp}g^{jq} u_{jpq} + o(\eps)\\
&= x_k - \eps\left(\nabla_g G(x)\right)_k + o(\eps), \quad \text{by \eqref{eq:gradG}}.
\end{split}
\]

The remaining terms may be expanded as
\[
\begin{split}
-\epsilon \E\left[  \nabla_g(V + G)(Y_*) \right]= -\epsilon \left[  \nabla_g(V + G)(x) \right] + o(\eps).  
\end{split}
\]
Adding all the terms together we get 
\[
\begin{split}
\E(Z) &= x - \epsilon\nabla_g\left[ V +2G\right]+ o(\epsilon)= x - \epsilon \frac{\partial F}{\partial x^*}(x^*) + o(\eps),\quad \text{by \eqref{eq:jacobian}.}
\end{split}
\]

Finally, let us estimate $\Cov(r_\eps)$. Since $Z \sim r_\eps(\cdot \mid x)$, by an abuse of notation,
\[
\begin{split}
\Cov(Z)&= \E \Cov(Z \mid Y) + \Cov\left( \E(Z \mid Y) \right)\\
&= \epsilon \E \left[\nabla^2 u^*(Y)\right] + \Cov(Y_*) + o(\epsilon) 
= \epsilon \frac{\partial x}{\partial x^*} +  \Cov(Y_*) + o(\epsilon). 
\end{split}
\]
The term $\Cov(Y_*)$ may again be estimated from the Taylor expansion \eqref{eq:taylorstar}. 
\[
\Cov(Y_*)= \eps\frac{\partial x}{\partial x^*}\frac{\partial x^*}{\partial x}\frac{\partial x}{\partial x^*} + o(\eps) = \epsilon \frac{\partial x}{\partial x^*} + o(\eps). 
\]
Thus, $\Cov(Z)=2\eps\frac{\partial x}{\partial x^*} + o(\eps)$. This completes the proofs of Steps 1 and 2 used to argue Theorem \ref{thm:mcconv}.
\end{proof}    

\begin{lemma}\label{lem:fourthmoment}
    Fix $x$, and let $Z \sim r_\eps(\cdot \mid x)$. Then  $\E\norm{Z-x}^3 \le C_0 \eps^{3/2}$ and $\E \norm{Z-x}^4 \le C_0\eps^2$, for some positive constant $C_0$. 
\end{lemma}

\begin{proof}
    This is a very similar argument to the last proof. We only give an outline. 
    If we compute exponential moments of $\hat{q}_\eps$ following an extension of the proof of Lemma \ref{lem:genfirstmoment}, it shows that $\hat{q}_\eps(\cdot y)$ is uniformly sub-Gaussian with a sub-Gaussian parameter $O(\eps)$. Hence, it implies the moment bounds $\E \norm{Z-Y_*}^3=O(\eps^{3/2})$ and  $\E\norm{Z-Y_*}^4=O(\eps^2)$. Since $Y$ is itself Gaussian, the rest of the argument follows exactly as in the above proof of Theorem \ref{thm:mcconv} steps 1 and 2. 
\end{proof}

\begin{proof}[Proof of Theorem \ref{thm:mcconv} step 3.] 
Fix $x$ and let $Z \sim r_\eps(\cdot \mid x)$. By lemma \ref{lem:fourthmoment}, there is a constant $C_0$ such that $\E\norm{Z - x}^4 \le C_0 \eps^2$, uniformly for all $k$. Thus, by Markov's inequality, 
    \[
    \eps^{-1}P\left(  \norm{X_1^\epsilon - x} > \delta \right)=\eps^{-1}P\left(  \norm{Z - x} > \delta \right) \le C_0\delta^{-4} \eps. 
    \]
    Taking $\eps \downarrow 0+$ completes the argument and the proof of Theorem \ref{thm:mcconv} is now complete. 
\end{proof}

\begin{proof}[Proof of Theorem \ref{thm:mclya}] 
By the Markov property it suffices to prove that a single step in the Markov chain is a contraction in the chi-square divergence with appropriate rate. That is, we need to show tha $\chi^2(p_1^\eps \mid \mu) \le (1 - c_0 \eps) \chi^2(p_0^\eps\mid \mu)$, for all $p_0^\eps$. This sort of inequality is known as a \textit{Strong Data Processing Inequality} (SDPI) and we use functional inequalities developed in that literature \cite{Raginsky2014StrongDP} combined with some comparison techniques.

Recall from \eqref{eq:whatispi} that $\pi_\eps(x,y)= e^{-V(x)} q_\eps(x,y)$. Consider the triplet $(X,Y,Z)$ defined below Definition \ref{defn:mcdef}. In particular, $(X, Z)$ is an exchangeable pair of random variables with the same law and the conditional density of $Z$, given $X=x$, is $r_\eps(\cdot \mid x)$, which is our Markov transition density. 

Consider any function $\xi\in \mathbf{L}^2(\mu)$.
Define the discrete Dirichlet energy \cite[page 46]{Raginsky2014StrongDP} of this Markov chain by the expression
\[
\mathcal{E}(\xi):= \frac{1}{2}\E (\xi(X) - \xi(Z))^2
\]
Say that the Markov chain satisfies a Poincar\'e inequality with a constant $c> 0$ if, for all $\xi\in \mathbf{L}^2(\mu)$, 
\begin{equation}\label{eq:poincare}
\Var(\xi) \le c \mathcal{E}(\xi). 
\end{equation}
The quantity $\lambda = c^{-1}$ is frequently called the spectral gap of this reversible Markov chain. We prove below that a Poincar\'e inequality holds for this Markov chain with a constant of the order $\Theta(1/\eps)$. 

Now, suppose $X'$ has density $\mu_0$, and $Z'$, given $X'=x'$, is distributed according to density $r_\eps(\cdot \mid x')$. Let $\mu_1$ be the marginal density of $Z'$. Of course, if $\mu_0=\mu$, then $\mu_1=\mu$ as well. 
Define the chi-square contraction rate as 
\[
\eta_{\chi^2}:=\sup_{\mu_0:\; \chi^2(\mu' \mid \mu) < \infty} \frac{\chi^2(\mu_1 \mid \mu)}{\chi^2(\mu_0 \mid \mu)}.
\]
If $\eta_{\chi^2} < 1$, the Markov chain is contractive in chi-square divergence, once it starts from an initial density that has a finite chi-square divergence with respect to $\mu$. However, it is known that $\eta_{\chi^2}=1-\lambda$ (see, for example, \cite[Theorem 4.3]{Raginsky2014StrongDP}). Hence, if we show a positive spectral gap, it implies an exponential rate of convergence of the Markov chain in the chi-square divergence. In fact, we show that the spectral gap is positive and $\Theta(\eps)$. As a consequence we obtain a chi-square contraction rate of $1-c_0\eps$, as desired. 

The remainder of this proof shows that the Poincar\'e inequality holds with a constant $C\ge c_0 \eps^{-1}$, for some constant $c_0$ independent of $\eps$. This is achieved by a comparison technique for discrete Dirichlet energies. 
\medskip

\noindent\textbf{Step 1.} The first step is to do a change of variables $x \mapsto x^*$. Consider instead the chain $(X^*, Z^*)$. Thus $X^* \rightarrow Z^*$ is a Markov chain with invariant distribution $e^{-F}= \left(\nabla u \right)_{\# \mu}$. We claim that if the chain $(X^*, Z^*)$ satisfies a Poincar\'e inequality, then so does $(X,Z)$. 

To see this, take any $\xi \in \mathbf{L}^2(e^{-V})$. Then $\xi \circ \nabla u^* \in \mathbf{L}^2(e^{-F})$. In fact the variance of $\xi$ under $e^{-V}$ is equal to the variance of $\xi \circ \nabla u^*$ under $e^{-F}$. If $(X^*, Z^*)$ satisfy Poincar\'e inequality with a constant $c$ then 
\[
\var(\xi \circ \nabla u^*) \le \frac{c}{2}\E\left( \xi \circ \nabla u^* (Z^*) -  \xi \circ \nabla u^* (X^*) \right)^2=\frac{c}{2}\E\left( \xi(Z) -  \xi(X) \right)^2. 
\]
This proves the Poincar\'e inequality for $(X,Z)$ with the same constant.
\medskip

\noindent\textbf{Step 2.} Next, let us derive the joint density of the random variables $(X^*, Y, Z^*)$ where $(X,Y,Z)$ is the triplet defined below Definition \ref{defn:mcdef}. 

The joint density $\eta_\eps(x,y,z)$ of $(X,Y, Z)$ is given by the product of the densities $\pi_\eps(x,y)\hat{q}_\eps(z\mid y)$. That is, if $G(x)=\frac{1}{2} \log \det g(x)$, as before, 
\[
\begin{split}
\eta_\eps(x,y,z)&= \frac{1}{(2\pi \eps)^{d/2}} \exp\left[ -V(x) - G(x) - \frac{1}{2\eps} (y-x^*)^T \frac{\partial x}{\partial x^*}(y-x^*)\right] \\
&\times \frac{1}{C_\eps(y)}\exp\left[-V(z)  - G(z) - \frac{1}{2\eps} (y-z^*)^T \frac{\partial z}{\partial z^*}(y-z^*) \right].
\end{split}
\]

For simplicity let $A=X^*, B=Z^*$ and let $\eta^*_\eps$ denote the joint density of $(A,Y, B)$. Then, by the change of variable formula \eqref{eq:jacobian}, if $x^*=a$ and $z^*=b$,
\[
\begin{split}
\eta^*_\eps(a,y,b)&= \frac{1}{(2\pi \eps)^{d/2}} \exp\left[ -F(a) - G(a_*) - \frac{1}{2\eps} (y-a)^T \frac{\partial a_*}{\partial a}(y-a)\right] \\
&\times \frac{1}{C_\eps(y)}\exp\left[-F(b)  - G(b_*) - \frac{1}{2\eps} (y-b)^T \frac{\partial b_*}{\partial b}(y-b) \right].
\end{split}
\]

Now, by our assumption, for some $c_1, \sigma_0 > 0$, $\nabla^2 u \ge \sigma^2_0 I$ and $\sup  \det g \le  c_1$. Hence,
\[
\eta^*_\eps(a,y,b)\ge \frac{C_\eps(y)^{-1}}{(2c_1\pi \eps)^{d/2}} \exp\left[ -F(a)  - \frac{1}{2\sigma_0^2\eps} \norm{y-a}^2 -F(b) - \frac{1}{2\sigma_0^2 \eps} \norm{y-b}^2 \right].
\]
By Lemma \ref{lem:gennormconstant}, there exists a constant $c_3>0$, such that for all $\eps\in (0,1]$, with $\phi= V+ G$, $x=y_*$ (i.e., $y= x^*$), and $H=\frac{\partial y}{\partial y_*}= \frac{\partial x^*}{\partial x}$,
\[
\eps^{-d/2} C_\eps(y) \le \frac{1}{c_3} \frac{e^{- V(x) - G(x)}}{\left[\det  \frac{\partial x^*}{\partial x}\right]^{1/2}}= \frac{1}{c_3} e^{-F(y)}.
\]
Thus, for some constant $c_4 >0$,
\[
\eta^*_\eps(a,y,b)\ge \frac{c_4}{(2\pi \eps)^{d}} \exp\left[ F(y) -F(a)  - \frac{1}{2\sigma_0^2\eps} \norm{y-a}^2 -F(b) - \frac{1}{2\sigma_0^2 \eps} \norm{y-b}^2 \right].
\]
\medskip

\noindent\textbf{Step 3.} Now comes the main comparison. Let $\tilde{\pi}_\eps(a,y)$ denote the joint density 
\begin{equation}\label{eq:gaussianpi}
\tilde{\pi}_\eps(a,y) =  \frac{1}{(2\pi \sigma_0^2 \eps)^{d/2}} \exp\left( -F(a)  - \frac{1}{2\sigma_0^2\eps} \norm{y-a}^2 \right).
\end{equation}
That is, sample $A$ from density $e^{-F}$ and $Y$, given $A=a$, is just Gaussian with mean $a$ and covariance $\sigma_0^2 I$. Let $e^{-\tilde{F}}$ denote the density of the $Y$ coordinate under $\tilde{\pi}_\eps$. 

Suppose now sample $(A, Y, B)$ according to the two-step Gibbs sampler run according to the joint distribution $\tilde{\pi}$. That is, $Y$, given $A$, and $B$, given $Y$, are samples from the two conditional densities of $\tilde{\pi}_\eps$. Then their joint density $\tilde{\eta}_\eps$ is given by 
\[
\tilde{\eta}_\eps(a,y,b)= \frac{1}{(2\pi \sigma_0^2 \eps)^{d}} \exp\left[ \tilde{F}(y) -F(a)  - \frac{1}{2\sigma_0^2\eps} \norm{y-a}^2 -F(b) - \frac{1}{2\sigma_0^2 \eps} \norm{y-b}^2 \right].
\]
Hence, for some constant $c_5 >0$, 
\[
\eta^*_\eps(a,y,b) \ge c_5 e^{F(y) - \tilde{F}(y)}\tilde{\eta}_\eps(a,y,b).
\]
We now show that $e^{F(y) - \tilde{F}(y)}$ is bounded below by a positive constant. To see this, note that, by definition 
\[
e^{-\tilde{F}(y)}= \E\left[ e^{-F(y - \sigma_0 \sqrt{\eps} Z)}\right], \quad Z \sim N(0, I).
\]
Therefore, 
\[
\begin{split}
e^{F(y) - \tilde{F}(y)}&= \E\left[ e^{ F(y) -F(y - \sigma_0 \sqrt{\eps} Z)}\right].
\end{split}
\]
By our assumption, the function $F$ is $L$-smooth. Hence, 
\[
F(y - \sigma_0 \sqrt{\eps} Z) \le F(y) - \sigma_0 \sqrt{\eps} Z \cdot \nabla F(y) + \frac{L}{2} \sigma_0^2 \eps \norm{Z}^2. 
\]
By the above and by Jensen's inequality 
\[
\begin{split}
e^{F(y) - \tilde{F}(y)}&\ge \E\left[ e^{ F(y) -F(y - \sigma_0 \sqrt{\eps} Z)}\right]\\
&\ge \E \left[ e^{\sigma_0 \sqrt{\eps} Z \cdot \nabla F(y) - \frac{L}{2} \sigma_0^2 \eps \norm{Z}^2}\right] \ge e^{-\frac{L}{2} \sigma_0^2 \eps d}\ge e^{-\frac{L}{2} \sigma_0^2 d},
\end{split}
\]
for all $\eps \in (0,1]$. All combined, for some positive constant $C_0$, $\eta^*_\eps(a,y,b) \ge C_0 \tilde{\eta}_\eps(a,y,b)$. By integrating out $y$ from both sides we get $\gamma^*_\eps(a,b) \ge C_0 \tilde{\gamma}_\eps(a,b)$, where $\gamma^*_\eps$ and $\tilde{\gamma}_\eps$ are the joint densities of $(A,B)$ under $\eta^*_\eps$ and $\tilde{\eta}_\eps$, respectively. Then, for any $\xi \in \mathbf{L}^2(e^{-F})$, 
\[
\E_{\gamma^*_\eps} \left( \xi(A) - \xi(B)\right)^2 \ge C_0 \E_{\tilde{\gamma}_\eps} \left( \xi(A) - \xi(B)\right)^2.
\]
If we now show that $\frac{c}{2}\E_{\tilde{\gamma}_\eps} \left( \xi(A) - \xi(B)\right)^2 \ge \Var(\xi)$, then the Poincar\'e gets transferred to $\gamma^*_\eps$ as well with the Poincar\'e constant given by $c/C_0$. 
\medskip

\noindent\textbf{Step 4.} All now remains to show is that $\tilde{\gamma}_\eps$ satisfies a Poincar\'e inequality with a constant that is $\Theta(1/\eps)$. This, however, follows from existing work \cite{KlartagOrdentlich}. 
Let us explain how by bringing in the Hirschfeld-Gebelein-R\'{e}nyi maximal correlation \cite[eqn. (8), (9)]{KlartagOrdentlich}. Given a pair of random variables $(X,Y)$ with a joint density $\eta$, the maximal correlation is defined as  
\[
S(\eta)= \sup_{f,g}\frac{ \E f(X) g(Y) - \E f(X) \E g(Y)}{\sqrt{\Var(f(X)) \Var(g(Y))}}
\]
It can be show that $S^2=\eta_{\chi^2}$, the chi-square contraction coefficient for the induced Markov chain whose transition density is given by the conditional density of $Y$, given $X$. This is also true, by symmetry of $S$, for the induced Markov chain whose transition density is given by the conditional density of $X$, given $Y$. See a derivation in the proof of Theorem 4.3 in \cite{Raginsky2014StrongDP}.

We claim that it suffices to show that that the maximal correlation $S^2$ or the $\eta_{\chi^2}$ corresponding to the joint density $\tilde{\pi}_\eps$ from \eqref{eq:gaussianpi} is strictly less than one. This is because $\tilde{\gamma}_\eps$ is simply two steps of the Gibbs sampler chain run according to the joint density $\tilde{\pi}_\eps$. Thus, the corresponding $\eta_{\chi^2}$ is the square of the chi-square contraction coefficient corresponding to the $\tilde{\pi}_\eps$. Basically that if one step of the Gibbs sampler (i.e., from $A\rightarrow Y$) is a contraction in $\chi^2$ divergence then so is the two-step (i.e. from $A \rightarrow Y\rightarrow B$) with the corresponding contraction coefficient being the square of the one-step coefficient.  

However, Theorem 1.1 in \cite{KlartagOrdentlich} computes the maximal correlation coefficient for the joitn distribution of the pair $(X, X + \sqrt{s} Z)$, where $X \sim \nu$ and $Z$ is standard multivariate normal and $s >0$ is a parameter. The joint distribution of $(X, X+ \sqrt{s} Z)$ is precisely $\tilde{\pi}_\eps$ when $s=\sigma^2_0 \eps$. Since we have assumed that $\nu=e^{-F}$ satisfies a Poincar\'e inequality with constant $c_F$, \cite[Theorem 1.1]{KlartagOrdentlich} applies to $\tilde{\pi}_\eps$ and gives an upper bound on the maximal correlation $S$ as $S^2\le (1 + \sigma_0^2 \eps / c_F)^{-1}$. Thus, by squaring, one gets that the chi-square contraction rate for the one step Markov chain $A \rightarrow B$ under $\tilde{\gamma}_\eps$ is bounded above by 
\[
\eta_{\chi^2} \le (1 + \sigma_0^2 \eps / c_F)^{-2}= 1 + \Theta(\eps). 
\]
Since the spectral gap is $1- \eta_{\chi^2}$, we get that the spectral gap under $\tilde{\gamma}_\eps$ is $\Theta(\eps)$. By our previous steps, original Markov chain, therefore, has a spectral gap of $\Theta(\eps)$ and a chi-square contraction rate of $(1- \Theta(\eps))$. 
This completes the proof of the theorem. 
\end{proof}

\section{Appendix}

\subsection{The special case of one dimension}

One particular case when we can say a lot more about the $\CD(\lambda, \infty)$ condition is in the case of dimension $d=1$. Here, $g(x) = u''(x)$, for some strictly convex potential $u$ and $\Ric_g = 0$ everywhere. Hence, the CD condition \eqref{eq:curvature-dimension} gets simplified to $\Hess_g(V)\succcurlyeq \lambda u''(x)$. This can be verified with the help of an interesting connection with Schwarzian derivatives \cite[Chapter 10]{hille1997ordinary}.

To wit, the only Christoffel symbol in one dimension is $\Gamma_{11}^1(x) = \frac{1}{2} \frac{g'(x)}{g(x)} = \frac{1}{2} \frac{u^{(3)}(x)}{u''(x)} $. For $f \in \mathcal{C}^2(\R)$, the Hessian (in affine coordinates) is 
\[
(\Hess_g f)(x) = f''(x) - \Gamma_{11}^1(x) f'(x). 
\]
Also, $G(x) = \frac{1}{2} \log(u''(x))$, so $G'(x) = \frac{1}{2} \frac{u^{(3)}(x)}{u''(x)} = \Gamma_{11}^1(x)$.
So,
\begin{align*}
    (\Hess_g G)(x) &= \frac{1}{2} \left[ \left( \frac{u^{(3)}(x)}{u''(x)} \right)' - \frac{1}{2} \left(\frac{u^{(3)}(x)}{u''(x)} \right)^2 \right] = \frac{1}{2} \left( Su' \right)(x),
\end{align*}
where $Sf$ is the Schwarzian derivative of $f$ defined as
\[
Sf(z)=\left(\frac{f''(z)}{f'(z)} \right)' - \frac{1}{2}\left( \frac{f''(z)}{f'(z)} \right)^2.
\]
Although $z$ is typically taken to be a complex variable, we will restrict ourselves to a real argument.

Our main result is the following.

\begin{theorem}\label{thm:negsign}
    Assume that V is increasing if and only if $u''$ is decreasing. Moreover, 
    \begin{equation}\label{eq:convexity1d}
    V''(x) \ge \lambda u''(x) - \frac{1}{2} (Su')(x), \quad \forall\; x \in \rr. 
    \end{equation}
    Then $(\mani, g, e^{-V})$ satisfies $\CD(\lambda, \infty)$ condition. 
\end{theorem}
   
\begin{proof}
The condition for $\CD(\lambda, \infty)$ is 
\begin{align}
     & V''(x) - \frac{1}{2} \frac{u^{(3)}(x)}{u''(x)} V'(x) + \frac{1}{2} \left( Su' \right)(x) \geq \lambda  u''(x) \\
    &\iff V''(x) -\frac{1}{2} (\log u''(x))'V'(x) + \frac{1}{2} \left( Su' \right)(x) \geq \lambda u''(x).\label{eq:convex1d}
\end{align}
Under our assumption  $(\log u''(x))' V'(x) \le 0$.
Hence, \eqref{eq:convexity1d} is a sufficient condition for \eqref{eq:convex1d}. This completes the proof. 
\end{proof}

It is not immediate if Theorem \ref{thm:negsign} provides any substantial simplification. We now show through various examples how to use it. 

 The following lemma is a slight adaptation of a well-known result \cite[Theorem 10.1.1]{hille1997ordinary}. We skip the proof. 

\begin{lemma}\label{lem:ODEschwarz}
Let $y_1, y_2$ denote the two linearly independent solutions of the ODE 
\[
y''(x) + q(x) y(x) = 0.
\]
Then if let $u'(x):= \frac{y_1(x)}{y_2(x)}$, then $S(u')=2q$. Moreover, there is always some way of indexing $y_1, y_2$ such that $u$ is convex. 
\end{lemma}




For the rest of this section we assume that $u'=\frac{y_1}{y_2}$ as in Lemma \ref{lem:ODEschwarz} for some convex $u$ and some suitable choice of $q$.


\begin{example} Take $q$ to be the constant function $-1$. Then $y_1=e^x$ and $y_2=e^{-x}$ are linearly independent solutions of $y''=y$. Hence, by defining $u$ via its derivative $u'=y_1/y_2=e^{2x}$, one gets $S(u')=-2$. Clearly, $u(x)=\frac{1}{2}e^{2x}$ is valid convex solution. 

Theorem \ref{thm:negsign} assumes that $V$ must be decreasing on $(0,\infty)$ and increasing on $(-\infty, 0)$. The condition \eqref{eq:convexity1d} becomes $V''(x) \ge 2\lambda e^{2x}  + \frac{1}{2}$. Consider the case of $\lambda=0$, where we simply demand $V''(x) \ge \frac{1}{2}$. Clearly, there is no such $V$ whose domain is the entire $\rr$. However, there are examples on compact intervals, such as $V(x)=(\abs{x}-1)^2$ on the interval $[-1,1]$. This functions is symmetric, decreasing on $(0,1]$, hence increasing on $[-1,0)$ and, on $[-1,1]\backslash\{0\}$, $V''(x)=2\ge 1/2$.
\end{example}



\begin{example}
For an example of the reverse phenomenon, consider $q(x)= 1/(4x^2)$. Two linearly independent  solutions to the Cauchy-Euler ODE
\[
4x^2 y'' + y =0, \quad x > 0, 
\]
are $y_1=\sqrt{x} \log x$ and $y_2=\sqrt{x}$. We indexed $y_1, y_2$ such a way that the Wronskian $W=1 >0$. Thus $u'(x)=\log x$ gives us the convex function $u(x)=x\log x - x$ for which $u''(x)=1/x$, and $S(u')=2q(x)=1/(2x^2)$. Restrict ourselves to the domain $(0, 1)$. 

The assumption in Theorem \ref{thm:negsign} can now be written as $V$ must be increasing on $(0, 1)$ and that
$V''(x) \ge \frac{\lambda}{x} - \frac{1}{x^2}$.
Note that, unlike the previous example, the RHS above does not have to be always positive which allows for concave choices of $V$. For example, take $\lambda=0$. Then we want $V''(x) \ge -\frac{1}{x^2}$. Let $V(x)=x^\alpha$ for some $\alpha \in (0,1)$, which makes $V$ a concave function of $x$ while being increasing in $(0,1)$. For this choice,
\[
V''(x)= -\frac{\alpha(1-\alpha)}{x^{2-\alpha}}\ge -\frac{1}{x^{2-\alpha}} \ge -\frac{1}{x^2},
\]
since $x^{2-\alpha} \ge x^2$,  and $\alpha(1-\alpha) \le 1$, for $x\in (0,1)$ and any $\alpha \in (0,1)$. 

Thus, for $\alpha \in (0, 1)$ we get a concave function $V(x)=x^\alpha$ for which Theorem \ref{thm:negsign} holds true. 
\end{example}

In all these examples, the domain is a proper subset of $\rr$. It is possible to define MLD in one dimension restricted to a subset of $\rr$. In that case the CD condition implies exponential convergence in Kullback-Leibler. However, we do not take up this construction here.



\subsection{Remaining proofs}

\begin{proof}[Proof of Theorem \ref{prop:LSI}] Use the final bullet point in \cite[Proposition 4.1]{CG16} for $F(u)=\log u$. Since we have already shown that $\cdc_D$ (and thus $\cdc$) satisfies a Poincar\'e inequality, LSI is implied by a defective LSI (\cite[(HFS4defect)]{CG16}, for the choice of $F(u)=\log u$). By the first bullet point in  \cite[Proposition 4.1]{CG16}, in order to prove a defective LSI, it suffices to prove a super Poincar\'e inequality (SPI). The proof below follows closely the argument in \cite[Proposition 3.5]{CG16}, except for a change of variable.

As in the proof of Theorem \ref{prop:PI}, from \eqref{eq:duallyapunov}, the dual MLD admits a Lyapunov function $\xi_\alpha^* \ge 1$ such that 
\begin{equation}\label{eq:globlya}
-\frac{1}{\xi_\alpha^*(y)}\genmld \xi_\alpha^*(y) \ge  \alpha (\varphi^*_\alpha(y_*) - d)\ge \alpha \delta V(y_*)-b_1,  
\end{equation}
for some $b_1 >0$. 
Here, as usual $y_*= \nabla u^*(y)$ is the primal coordinate corresponding to $y$, and the final inequality is due to assumption (c).

For every $\lambda \ge 0$, let $A_\lambda:=\{y\in \rr^d:\; \alpha \delta V(y_*) \le b_1 + \lambda\}$. 
By our assumptions, $A_0$ is compact. Let 
$b_0:=\sup_{A_0} \left( - \alpha \delta V(y_*) + b_1 \right)$.

Recall that $\nu=e^{-F}$ is the invariant distribution for the dual MLD. Let $M:=\sup_y(-F(y)) < \infty$. Let $\gamma$ be a smooth test function.  
Then, for any $\lambda >0$, 
\[
\begin{split}
\int \gamma^2(y) d\nu(y) &= \int_{A_\lambda} \gamma^2(y) e^{-F(y)}dy + \int_{A^c_\lambda} \gamma^2(y) d\nu\\
&\le e^{M} \int_{A_\lambda} \gamma^2(y)dy + \frac{1}{\lambda}\int_{A^c_\lambda} (\alpha \delta V(y_*) - b_1) \gamma^2(y) d\nu.
\end{split}
\]
By construction $\int_{A^c_\lambda\cap A_0} (\alpha \delta V(y_*) - b_1)\gamma^2(y) d\nu \ge 0$. Thus the second integral on the RHS may be bounded above by
\[
\begin{split}   
& \frac{1}{\lambda}\int_{\rr^d} (\alpha \delta V(y_*) - b_1) \gamma^2(y) d\nu- \frac{1}{\lambda}\int_{A_0} (\alpha \delta V(y_*) - b_1) \gamma^2(y) d\nu \\
&\le \frac{1}{\lambda}\int_{\rr^d} (\alpha \delta V(y_*) - b_1) \gamma^2(y) d\nu+ \frac{e^M b_0}{\lambda}\int_{A_0} \gamma^2(y) dy\\
&\le \frac{1}{\lambda}\int_{\rr^d} (\alpha \delta V(y_*) - b_1)\gamma^2(y) d\nu+ \frac{e^M b_0}{\lambda}\int_{A_\lambda} \gamma^2(y) dy.
\end{split}
\]
Thus, combining this bound with the one above gives
\begin{equation}\label{eq:lsiproof}
\begin{split}
    \int \gamma^2 d\nu &\le e^{M}\left( 1 + \frac{b_0}{\lambda} \right) \int_{A_\lambda} \gamma^2(y)dy + \frac{1}{\lambda}\int_{\rr^d} (\alpha \delta V(y_*) - b_1) \gamma^2(y) d\nu.
\end{split}
\end{equation}
Now we bound the two integrals on the RHS separately. For the second integral, by \eqref{eq:globlya}, 
\[
\begin{split}
  \frac{1}{\lambda}  \int (\alpha \delta V(y_*) - b_1) \gamma^2(y) d\nu &\le  \frac{1}{\lambda}  \int -\frac{1}{\xi_\alpha^*(y)}\genmld \xi_\alpha^*(y) \gamma^2(y) d\nu = \frac{1}{\lambda} \cdc_D(\gamma).
\end{split}
\]
The last equality is due to integration-by-parts. See, for example, a very similar calculation done in the first display on \cite[page 64]{BPCG08}. 

Continue to follow the proof of \cite[Proposition 3.5]{CG16} and bound the first integral on the RHS of \eqref{eq:lsiproof} by a local super Poincar\'e inequality. For any $s>0$, 
\[
\int_{A_\lambda} \gamma^2(y)dy\le s \int_{A_\lambda} \norm{\nabla \gamma}^2dy + \beta(s) \left(\int_{A_\lambda} \abs{\gamma} dy\right)^2, 
\]
where $\beta(s)=C_2(s^{-d/2} +1)$, for some absolute constant $C_2$. This follows from \cite[Proposition 3.1]{CGWW}. 





Thus combining these bounds, we get that for all $s>0$, $\int \gamma^2 d\nu\le $
\begin{equation}\label{eq:lsiproof2}
  e^{M}\left( 1 + \frac{b_0}{\lambda} \right) \left(s \int_{A_\lambda} \norm{\nabla \gamma}^2dy + \beta(s) \left(\int_{A_\lambda} \abs{\gamma} dy\right)^2\right) + \frac{1}{\lambda} \cdc_D(\gamma).
\end{equation}
By our assumption $(\nabla^2 u^*)^{-1}$ has a global upper bound. Thus, for suitable positive constants, $\int_{A^c_\lambda} \norm{\nabla \gamma}^2dy \le C' \cdc_D(\gamma)$. Hence, 
\[
\int \gamma^2 d\nu\le \left( C'e^{M}\left( 1 + \frac{b_0}{\lambda} \right) s + \frac{1}{\lambda} \right) \cdc_D(\gamma) + \beta(s) \left(\int_{A_\lambda} \abs{\gamma} dy\right)^2.
\]

Now, consider a change of variables from $y \mapsto x=y_*$. The pre-image of the set $A_\lambda$ is the set $B_\lambda:=\{x: \alpha \delta V(x) \le b_1 + \lambda\}$. 
Thus, for some suitable positive constants $C_i$, $i=1,2,3$ and $c'$, 
\[
\begin{split}
&\left(\int_{A_\lambda} \abs{\gamma} dy\right)^2= \left( \int_{B_\lambda} \abs{\gamma}(x^*) \det \nabla^2 u(x) dx \right)^2= C_1 \left( \int_{B_\lambda} \abs{\gamma}(x^*)  dx \right)^2\\
&\le C_3 e^{c'\lambda} \left( \int_{\rr^d} \abs{\gamma}(x^*) e^{-V(x)} dx \right)^2=C_3 e^{c'\lambda} \left( \int_{\rr^d} \abs{\gamma}(y) e^{-F(y)} dy \right)^2. 
\end{split}
\]

Hence, 
\[
\int \gamma^2 d\nu\le \left( C'e^{M}\left( 1 + \frac{b_0}{\lambda} \right) s + \frac{1}{\lambda} \right) \cdc_D(\gamma) + C_3 s^{-d/2} e^{c'\lambda} \left( \int_{\rr^d} \abs{\gamma} d\nu \right)^2.
\]
Now, pick $\lambda=c/s$ to get the following super Poincar\'e inequality for small $s$. For some suitable constants $c, C', c'>0$, to get that for all small enough $s>0$, 
\[
\int \gamma^2 d\nu\le c s \cdc_D(\gamma) + C' e^{c'/s} \left( \int_{\rr^d} \abs{\gamma} d\nu \right)^2.
\]
The rest of the argument goes through exactly as the argument in the proof of \cite[Proposition 3.5]{CG16} from eqn (3.6) and below. This completes the proof.
\end{proof}

\bibliographystyle{abbrv}
\bibliography{./biblio.bib}

\end{document}